\documentclass[10pt]{amsart}
\usepackage{amscd,amsmath,amssymb,amsfonts,amsthm}
\usepackage{mathrsfs}
\usepackage[cmtip, all]{xy}
\usepackage{graphicx}

\usepackage{color}
\definecolor{refkey}{gray}{.5}   
\definecolor{labelkey}{gray}{.5} 
\definecolor{Red}{rgb}{1,0,0}

 \definecolor{dark-red}{rgb}{0.4,0.15,0.15}
 \definecolor{dark-blue}{rgb}{0.15,0.15,0.4}
   \definecolor{medium-blue}{rgb}{0,0,0.5}
\usepackage{aliascnt}
 \usepackage{hyperref}
 \hypersetup{
     colorlinks, linkcolor=dark-red,
     citecolor=dark-blue, urlcolor=medium-blue
 }

\numberwithin{equation}{section}

\theoremstyle{plain}

\newaliascnt{theorem}{equation}  
\newtheorem{theorem}[theorem]{Theorem}  
\aliascntresetthe{theorem}

\newaliascnt{thm}{equation}  
\newtheorem{thm}[thm]{Theorem}  
\aliascntresetthe{thm}

\newaliascnt{proposition}{equation}  
\newtheorem{proposition}[proposition]{Proposition}
\aliascntresetthe{proposition}

\newaliascnt{prop}{equation}  
\newtheorem{prop}[prop]{Proposition}
\aliascntresetthe{prop}

\newaliascnt{lemma}{equation}  
\newtheorem{lemma}[lemma]{Lemma}
\aliascntresetthe{lemma}

\newaliascnt{lem}{equation}  
\newtheorem{lem}[lem]{Lemma}
\aliascntresetthe{lem}

\newaliascnt{corollary}{equation}  
\newtheorem{corollary}[corollary]{Corollary}
\aliascntresetthe{corollary}

\newaliascnt{cor}{equation}  
\newtheorem{cor}[cor]{Corollary}
\aliascntresetthe{cor}

\newaliascnt{claim}{equation}  

\aliascntresetthe{claim}

\newaliascnt{conjecture}{equation}  

\aliascntresetthe{conjecture}

 \theoremstyle{definition}

\newaliascnt{definition}{equation}  
\newtheorem{definition}[definition]{Definition}
\aliascntresetthe{definition}

\newaliascnt{defn}{equation}  
\newtheorem{defn}[defn]{Definition}
\aliascntresetthe{defn}

\newaliascnt{example}{equation}  
\newtheorem{example}[example]{Example}
\aliascntresetthe{example}

\newaliascnt{exm}{equation}  
\newtheorem{exm}[exm]{Example}
\aliascntresetthe{exm}

\newaliascnt{remark}{equation}  
\newtheorem{remark}[remark]{Remark}
\aliascntresetthe{remark}

\newaliascnt{remk}{equation}  
\newtheorem{remk}[remk]{Remark}
\aliascntresetthe{remk}  

\newtheorem*{ResP}{Resolution Property for {$X$}}

\newcommand{\aref}[1]{\autoref{#1}}

 \newtheorem{thm*}{Theorem}

\newcommand{\sB}{{\mathcal B}}
\newcommand{\sC}{{\mathcal C}}
\newcommand{\sD}{{\mathcal D}}
\newcommand{\sE}{{\mathcal E}}
\newcommand{\sF}{{\mathcal F}}

\newcommand{\sK}{{\mathcal K}}

\newcommand{\sM}{{\mathcal M}}
\newcommand{\sN}{{\mathcal N}}
\newcommand{\sO}{{\mathcal O}}
\newcommand{\sP}{{\mathcal P}}

\newcommand{\sS}{{\mathcal S}}

\newcommand{\sU}{{\mathcal U}}
\newcommand{\sV}{{\mathcal V}}
\newcommand{\sW}{{\mathcal W}}
\newcommand{\sX}{{\mathcal X}}
\newcommand{\sY}{{\mathcal Y}}
\newcommand{\sZ}{{\mathcal Z}}
\newcommand{\A}{{\mathbb A}}

\newcommand{\C}{{\mathbb C}}

\renewcommand{\P}{{\mathbb P}}
\newcommand{\Q}{{\mathbb Q}}
\newcommand{\R}{{\mathbb R}}

\newcommand{\Z}{{\mathbb Z}}

\newcommand{\Ho}{\mathrm{H}}
\newcommand{\mcal}[1]{\mathcal{#1}}
\newcommand{\iso}{\cong}

\newcommand{\surj}{\twoheadrightarrow}
\newcommand{\inj}{\hookrightarrow}

\newcommand{\Pic}{{\rm Pic}}

\newcommand{\Hom}{{\rm Hom}}

\newcommand{\Spec}{{\rm Spec \mspace{1mu}}}

\newcommand{\cof}{{\rm cof}}

\newcommand{\0}{\emptyset}
\newcommand{\sHom}{{\mathcal{H}{om}}}

\newcommand{\id}{{\operatorname{id}}}

\newcommand{\Sch}{{\operatorname{\mathbf{Sch}}}}

\newcommand{\holim}{\mathop{{\rm holim}}}

\newcommand{\<}{\langle}
\renewcommand{\>}{\rangle}
\newcommand{\Sets}{{\mathbf{Sets}}}

\newcommand{\Sm}{{\mathbf{Sm}}}

\newcommand{\Ab}{{\mathbf{Ab}}}

\newcommand{\ds}{{/\kern-3pt/}}

\newcommand{\Supp}{{\operatorname{Supp}}}

\newcommand{\colim}{\mathop{\text{colim}}}

\newcommand{\Th}{{\operatorname{\rm Th}}}

\newcommand{\ov}{\overline}

\newcommand{\wh}{\widehat}

\renewcommand{\dim}{\text{\rm dim}}

\newcommand{\tuborg}{\left\{\begin{array}{ll}}
\newcommand{\sluttuborg}{\end{array}\right.}

\newcommand{\Shv}{{\mathbf{Shv}}}

\newcommand{\cd}{\smash\cdot}

\newcommand{\twoarrow}{\mathrel{\substack{\textstyle\xrightarrow{\hspace*{ 0.8cm}}\\[-0.3ex]
                      \textstyle\xrightarrow{\hspace*{0.8cm}}}}}

\begin{document} 
\title{Motivic homotopy theory of group scheme actions}
\author[J. Heller]{Jeremiah Heller}
\email{jeremiahheller.math@gmail.com}
\address{Department of Mathematics, University of Illinois at Urbana-Champaign}
\author[A. Krishna]{Amalendu Krishna}
\email{amal@math.tifr.res.in}
\address{School of Mathematics, Tata Institute of Fundamental Research, 
 Mumbai, India}
\author[P.A. {\O}stv{\ae}r]{Paul Arne {\O}stv{\ae}r}
\email{paularne@math.uio.no}
\address{Department of Mathematics, University of Oslo, Norway}

\keywords{Equivariant and motivic homotopy theory, group schemes, equivariant vector bundles and $K$-theory}
 
\subjclass[2010]{14F42, 14L30, 55P91}

\begin{abstract}
We define an unstable equivariant motivic homotopy category for an algebraic group over a Noetherian base scheme. 
We show that equivariant algebraic $K$-theory is representable in the resulting homotopy category. 
Additionally, 
we establish homotopical purity and blow-up theorems for finite abelian groups.
\end{abstract}
\maketitle
\tableofcontents

\section{Introduction}\label{section:Intro}
There is a long and fruitful tradition of using homotopical ideas to study algebro-geometric invariants.  
In groundbreaking work \cite{MV}, Morel-Voevodsky introduced a full-fledged homotopy theory for smooth algebraic varieties. Their introduction of the motivic homotopy category has its roots in work of
Rost and Voevodsky resolving the Bloch-Kato conjectures on Milnor
$K$-theory and Galois cohomology \cite{Voevodsky:Z2,Voevodsky:Zl}.  
Since then, this framework has shown itself to be a useful setting for studying algebro-geometric cohomology theories and it has yielded many applications 
to the study of algebraic cycles, algebraic $K$-theory, and quadratic forms.

In recent years there has been a growing interest in equivariant homotopy theory, in both classical homotopy theory and in motivic homotopy theory. This owes in part to the recent  success of equivariant homotopy theory in work of Hill-Hopkins-Ravenel \cite{HHR} on the Kervaire invariant one problem.
Equivariant motivic homotopy theory for finite flat group scheme actions was first defined by Voevodsky in \cite{Deligne} in order to study motivic Eilenberg-MacLane spaces. It was then taken up by Hu-Kriz-Ormsby \cite{HKO} to study the homotopy limit problem in Hermitian $K$-theory.  
The equivariant motivic homotopy category provides a convenient setting for defining new  cohomology theories on smooth schemes equipped with the action by a group scheme $G$, as well as studying old ones.
For the group of order two, it has already been exploited to define new theories. Important examples are Real algebraic $K$-theory, Real motivic cobordism \cite{HKO},  
and a Bredon type theory of motivic cohomology \cite{HVO}.
An important classical example is provided by equivariant 
algebraic $K$-theory introduced by Thomason \cite{Thomason:GK} and shown to be representable in \aref{cor:M-Uns-Rep1} below.

To set the stage for our results, we briefly mention one of the motivating applications behind our work: Asok's program \cite{Asok} on the $\A^1$-contractibility of a certain three dimensional complex variety, the Russel cubic. This question has its roots in the
Zariski Cancellation problem\footnote{The Zariski Cancellation problem asks whether a smooth complex variety $X$ such that there is an isomorphism $X\times\A^1_{\C}\iso \A^{n+1}_{\C}$ must itself be isomorphic $\A^{n}_{\C}$.}. The Russel cubic is expected to be a counterexample to the Zariski Cancellation problem in dimension three. However, it is presently unknown whether the Russel cubic is stably isomorphic to $\A^{3}_{\C}$. 
Asok's program is an attempt to answer this question. One of the 
steps in this program requires that  equivariant algebraic $K$-theory is a \emph{fixed-point-equivalence invariant}, i.e., an equivariant map $f:X\to Y$ of smooth $G$-schemes should induce an isomorphism on equivariant algebraic $K$-theory whenever $f$ induces an $\A^1$-weak equivalence $X^{H}\to Y^{H}$ on the fixed point schemes for all subgroups $H\subseteq G$.
Classically,  equivariant topological $K$-theory is a fixed-point-equivalence invariant, which follows from the fact that it is representable in the equivariant homotopy category.  We do show that equivariant algebraic $K$-theory is representable in the equivariant motivic homotopy category but equivariant motivic equivalences are not detected by fixed points. In fact, 
Herrmann \cite{Herrmann} has shown that equivariant algebraic $K$-theory is not a 
fixed-point-equivalence invariant. Nonetheless, we show in \aref{thm:ieNrationalizedKtheory} that rational equivariant algebraic $K$-theory is a fixed-point-equivalence invariant. This suffices for Asok's program. In the sequel paper \cite{HKO}, building on this paper, the remaining steps in his program are checked. (Although it should be noted that not all steps of the program turn out to work, see loc.~cit.~ for details.)


Now we turn to describing our results. First, we extend and elaborate on the foundations of equivariant motivic homotopy theory, for a flat algebraic group scheme $G$.  
Our construction of the equivariant motivic homotopy category in \aref{section:UHC} follows the now familiar pattern. We start with the category of simplicial presheaves on smooth schemes with a $G$-action, equipped with a global model structure and then form the left Bousfield localization at suitable local equivalences, finally we  further localize to force the affine line to become contractible. The local equivalences take into account the equivariant Nisnevich topology, defined in 
\aref{section:CD-structure}  by a $cd$-structure on the category of smooth $S$-schemes equipped with a $G$-action. We show that this definition yields a topology equivalent to the one Voevodsky defines \cite{Deligne}. In particular the equivariant motivic homotopy category which we construct here agrees with the one previously constructed by Voevodsky, when $G$ is finite.

Equivariant algebraic $K$-theory, introduced by Thomason \cite{Thomason:GK} is the $K$-theory of $G$-vector bundles.
We show that it is representable in the equivariant motivic homotopy category
when the base scheme is regular and $G$ satisfies the resolution property, i.e., every coherent $G$-bundle is a quotient of a $G$-vector bundle. Reductive algebraic groups satisfy this property. See \aref{cor:M-Uns-Rep1} for a precise statement of representability.
\begin{theorem}
Let $S$ be a regular Noetherian base scheme and $G$ a flat algebraic group scheme over $S$ which satisfies the resolution property. 
Then equivariant algebraic $K$-theory for smooth $G$-schemes over $S$ is representable in the equivariant motivic homotopy category.
\end{theorem}

As an application, we show in \aref{subsection:ECAC} that if $G$ is a finite cyclic group over a field $k$,
then every $G$-equivariant vector bundle on an equivariantly $\A^{1}$-contractible 
smooth affine curve is the pullback of an equivariant vector bundle on $\Spec(k)$.
It is an open question whether the same holds in higher dimensions.

%
A surprising feature of equivariant motivic homotopy theory  is that equivariant motivic weak equivalences are {\sl not} detected by fixed points. To remedy this,
Herrmann \cite{Herrmann} constructed a variant of the equivariant motivic homotopy category, for finite groups over fields, 
using a different equivariant generalization of the Nisnevich topology, namely the fixed-point Nisnevich topology. The weak equivalences in the resulting homotopy category are maps $f:X\to Y$ of motivic $G$-spaces such that the induced map on fixed point loci $X^{H}\to Y^{H}$ is an $\A^1$-weak equivalences for all subgroups $H\subseteq G$.
However, as shown in loc.~cit., equivariant algebraic $K$-theory does not satisfy descent with respect to the fixed point Nisnevich topology, 
and therefore is not representable in the homotopy category which he constructs. Nonetheless, we show  in 
\aref{thm:ieNrationalizedKtheory} that these difficulties disappear if one considers instead equivariant algebraic $K$-theory with rational coefficients. It follows that rational equivariant algebraic $K$-theory is a fixed-point-equivalence invariant.
\begin{theorem}
Let $k$ be a field and $G$  a finite group.
Equivariant algebraic $K$-theory with rational coefficients 
satisfies descent in the fixed point Nisnevich topology on smooth $G$-schemes over $k$. 
\end{theorem}

%
The homotopy purity theorem \cite[Theorem~3.2.23]{MV} is a fundamental tool in motivic homotopy theory. In \aref{thm:Purity}, we establish the following equivariant generalization.
\begin{thm}
Let $k$ be a perfect field and $G$ be a finite abelian group whose order is prime to 
$\mathrm{char}(k)$. Suppose that $k$ contains a primitive $d$th root of unity, where $d$ is the least common multiple of the orders of elements of $G$.
Then for any closed immersion $Z\hookrightarrow X$  of smooth $G$-schemes over $k$ there is an equivariant motivic weak equivalence 
$X/{(X \setminus Z)} 
\simeq
{\rm Th}(N_{Z/X})$ of pointed motivic $G$-spaces.  
\end{thm}

Here $N_{Z/X}$ is the normal bundle and for an equivariant vector bundle $\sV$ over $Z$, ${\rm Th}(\sV) = \sV/(\sV \setminus Z)$ is the associated Thom space. 
We follow a strategy similar to Morel-Voevodsky's in \cite{MV} and argue that we can reduce to the case of a zero section of an equivariant vector bundle. In this case an easy deformation to the normal cone argument yields the theorem. 
 
Finally, besides the already mentioned antecedents \cite{Deligne, HKO} and alternate approach \cite{Herrmann} to our work, we mention that another alternate approach is carried out by Carlsson-Joshua \cite{CJ} originating in their work on Carlsson's conjecture relating algebraic $K$-theory of fields to representation theory.

{\sl Outline of the paper:}
We introduce the equivariant Nisnevich topology via a $cd$-structure on $G$-schemes in \aref{section:CD-structure} and  show that it is regular, complete, and bounded. In the case of a finite group, we provide alternate descriptions of the topology and identify the points explicitly. In \aref{section:Models} we recall the standard local model structures on presheaves of simplicial sets.  The motivic model structures are introduced 
in \aref{section:UHC} for unpointed and pointed presheaves. 
In \aref{section:Nis-Desc-rep} we show that equivariant algebraic $K$-theory  for smooth $G$-schemes is representable in the equivariant motivic homotopy category, when the base $S$ is regular and $G$ has the resolution property. As an application we characterize equivariant vector bundles, in the case of a cyclic group, on equivariantly contractible curves over a field. In \aref{subsection:H-Nis} we focus on descent for rational equivariant algebraic $K$-theory in the fixed point Nisnevich topology. Finally in \aref{section:EPT} we establish our equivariant homotopy purity theorem.

{\bf Notations:} 
Throughout $S$ will always be a separated Noetherian scheme of finite Krull dimension.
An $S$-scheme is a separated scheme of finite type over $S$. When the base is understood we will often refer to an $S$-scheme simply as a scheme.
We write 
$\Sch_S$ for this category and
 $\Sm_S$ for the full subcategory comprised of schemes which are smooth over $S$.

An algebraic group scheme $G\to S$ is a group object in $\Sch_{S}$. In particular, it is separated and of finite type.  
We always assume that $G\to S$ is flat, although we often impose additional assumptions as needed.
The category $\Sch^G_S$ of $G$-schemes over $S$ has as objects pairs $(X, \mu_X)$ consisting of an $S$-scheme $X$ and a left $G$-action $\mu_{X}:G\times_{S} X\to X$ over $S$ and  maps between $G$-schemes are maps of $S$-schemes which are  $G$-equivariant. 
Similarly we write $\Sm^G_S$ for the category of smooth $G$-schemes.

{\bf Remark:} 
We point out that using the same methods one may construct  a motivic homotopy theory for Deligne-Mumford stacks.
For an affine base scheme $S$ as above, 
one considers the category of Noetherian and separated Deligne-Mumford stacks of finite type over $S$ equipped with its Nisnevich topology introduced in \cite{AO}. The Nisnevich descent theorem for $K$-theory of Deligne-Mumford stacks \cite{AO} shows that $K$-theory is representable in this setting.

\section{Equivariant Nisnevich topology}
\label{section:CD-structure}
The equivariant Nisnevich topology was originally defined by Voevodsky \cite{Deligne} in order to study symmetric powers of motivic spaces. 
In this section we define the equivariant Nisnevich topology  
for a flat algebraic group scheme $G\to S$ via a $cd$-structure, which we show is regular, complete, and bounded in the sense of \cite{V:cd}.

\subsection{Equivariant Nisnevich \texorpdfstring{$cd$}{cd}-structure}
\label{subsection:Nis-N}

A {\sl distinguished equivariant Nisnevich square} is a cartesian square in $\Sch^G_S$
\begin{equation}\label{eqn:cd-square}
\xymatrix{
B \ar[r] \ar[d] & Y \ar[d]^{p} \\
A \ar@{^{(}->}[r]^{j} & X,} 
\end{equation}
where $j$ an open immersion, $p$ is \'etale, and 
$(Y\setminus B)_{\rm red} \to (X\setminus A)_{\rm red}$ is an isomorphism.
The collection of distinguished equivariant Nisnevich squares forms a $cd$-structure in the sense of \cite{V:cd}.
\begin{defn}\label{defn:dist-square}
The equivariant Nisnevich $cd$-structure on  $\Sch_{S}^{G}$  is the collection of distinguished equivariant Nisnevich squares  in  $\Sch_{S}^{G}$. 
\end{defn}

As we now show,  the equivariant Nisnevich $cd$-structure has good properties. The following is an equivariant analogue of \cite[Theorem~2.2]{V:Niscdh}. The proof is obtained by following the steps in the 
nonequivariant case with suitable modifications at various stages. 
We refer to \cite[\S~2]{V:cd} for the definition of a complete, regular, and bounded $cd$-structure. 

\begin{thm}\label{thm:Comparison}
The equivariant Nisnevich $cd$-structure on $\Sch^G_S$ is complete, regular, and bounded.
\end{thm}
\begin{proof}
The completeness assertion follows from \cite[Lemma~2.4]{V:cd} since distinguished equivariant Nisnevich squares are closed under pullbacks.
  
To prove regularity,
observe that for a distinguished equivariant Nisnevich square \eqref{eqn:cd-square}  there is an induced distinguished Nisnevich square 
by \cite[Theorem~2.2]{V:Niscdh} 
\begin{equation}\label{eqn:CB*1}
\xymatrix@C.8pc{
B \ar[r]^{e'} \ar[d]_{\Delta_B} & Y \ar[d]^{\Delta_Y} \\
B \times_A B \ar[r] & Y \times_X Y.}
\end{equation}
Because the maps in~\eqref{eqn:CB*1} are $G$-equivariant it is also a distinguished equivariant Nisnevich square. 
The regularity condition now follows from \cite[Lemma~2.11]{V:cd}.

The boundedness condition is not straightforward from the non-equivariant case.
First we define a density structure on $\Sch^G_S$. 
For $X \in \Sch^G_S$ and $i \ge 0$, 
let $D_i(X)$ denote the class of equivariant open embeddings $U \to X$  that define an element of the density structure on  
\cite[Proposition~2.10]{V:cd} under the forgetful functor $\Sch^G_S \to \Sch_S$.
That is, an equivariant open embedding $U\to X$ is in $D_{i}(X)$ provided
for every $z\in X\setminus U$ there exists a sequence of points $z=x_0,x_1,\dots,x_i$ in $X$ such that for $0\leq j<i$, 
$x_j\neq x_{j+1}$ and $x_j\in\overline{\{x_{j+1}\}}$. 
One verifies easily that this defines a density structure on $\Sch^G_S$, 
and it is locally of finite dimension. 

To prove boundedness, 
it is enough to show that every distinguished equivariant Nisnevich square is reducing with respect to the above density structure.
Consider a distinguished equivariant Nisnevich square of the form~\eqref{eqn:cd-square} and suppose $B_0 \in D_{i-1}(B), A_0 \in D_i(A)$ and $Y_0 \in D_i(Y)$.
Applying \aref{lem:Density} below to the morphism $j \coprod p$ we can find $X_0 \in D_i(X)$ such that $j(A_0) \cap p(Y_0) \subseteq X_0$.
Replacing $Y$ by $Y_0$, $A$ by $A_0$, $B$ by $B'= A_0 \times_{X} Y_0$, $X$ by $X_0$, and applying \cite[Lemma~2.5]{V:Niscdh} we are 
reduced to consider the distinguished equivariant Nisnevich square    
\begin{equation}\label{eqn:cd-square-1}
\xymatrix@C.9pc{
B' \ar[r] \ar[d] & Y_0 \ar[d]^{p} \\
A_0 \ar[r]_{j} & X_0.} 
\end{equation}

We now set 
\begin{equation*}\label{eqn:cd-square-2}
B'_0 = B' \cap B_0, \ Z = B' \setminus B'_0,  \ Y' = Y_0 \setminus 
cl_{Y_0}(Z),  \ A' = A_0 \ {\rm and} \ X' = j(A_0) \cup p(Y').
\end{equation*}
In \cite[Proposition~2.10]{V:Niscdh} it is noted that 
\begin{equation*}\label{eqn:cd-square-3}
\xymatrix@C.9pc{
B'_0 \ar[r] \ar[d] & Y' \ar[d]^{p} \\
A_0 \ar[r]_{j} & X'} 
\end{equation*}   
is a distinguished Nisnevich square which satisfies the required properties.
To complete the proof we observe that the inclusions in this square are $G$-invariant.
\end{proof}

\begin{lem}\label{lem:Density}
Let $f: X \to Y$ be a morphism in $\Sch^G_S$ and assume that there exists a $G$-invariant dense open subset $U$ in $Y$ 
such that $f^{-1}(U)$ is dense and $f^{-1}(U) \to U$ has fibers of dimension zero. 
Then for any $i \ge 0$ and $V \in D_i(X)$, 
there exists $W \in D_i(Y)$ such that $f^{-1}(W) \subseteq V$.
\end{lem}
\begin{proof}
By \cite[Lemma~2.9]{V:Niscdh}, 
there exists $W' \in D_i(Y)$ such that $f^{-1}(W') \subseteq V$. 
But $W'$ need not be $G$-invariant. 
Since $G\to S$ is flat, the orbit $G\cd W'\subseteq Y$ is open. Set $W=G\cd W'$.
Because $V \subseteq X$ is $G$-invariant (by definition of our density structure), 
it follows that $f^{-1}(W) \subseteq V$.
Furthermore, 
as $W' \subseteq W$ and $W' \in D_i(Y)$, 
we see that $W \in D_i(Y)$. 
This proves the lemma.
\end{proof}

\begin{defn}
The {\sl equivariant Nisnevich topology} is the Grothendieck topology associated to the 
equivariant Nisnevich $cd$-structure. Write $\Sch^{G}_{S/{\rm Nis}}$ and
$\Sm^G_{S/{\rm Nis}}$ for the respective categories of $G$-schemes and smooth $G$-schemes equipped with the equivariant Nisnevich topology. 
\end{defn}

\begin{remk}
For every distinguished equivariant Nisnevich square \eqref{eqn:cd-square}, 
the sieve generated by $j$ and $p$ is a covering of $X$, 
and the empty sieve covers the empty scheme.
\end{remk}

\begin{cor}\label{cor:Sheaf-Nis}
A presheaf $\sF$ is a  sheaf in the equivariant Nisnevich topology if and only if $\sF(\0)=pt$ and for any distinguished equivariant Nisnevich (\ref{eqn:cd-square}), the resulting square
$$
\xymatrix{
\sF(X) \ar[r]\ar[d] & \sF(A)\ar[d] \\
\sF(Y) \ar[r] & \sF(B)
}
$$
is cartesian. 
\end{cor}
\begin{proof}
This follows from \aref{thm:Comparison} and \cite[Lemma~2.9, Proposition~2.15]{V:cd}.
\end{proof}

\begin{cor}\label{cor:sub-can}
The equivariant Nisnevich topology is sub-canonical.
\end{cor}
\begin{proof}
It is straightforward to check that a representable presheaf takes a distinguished equivariant Nisnevich square to a cartesian square.
\end{proof}

Write $H^{i}_{GNis}(X, \mcal{F})$ for the $i$th sheaf cohomology group in the $G$-equivariant Nisnevich cohomology.
\begin{cor}\label{cor:Nis-Dim}
Let $\sF$ be a sheaf of abelian groups 
on $\Sch^G_{S/{\rm Nis}}$. Then  
$$
H^i_{GNis}(X, \sF) = 0
$$ 
for $i>\dim(X)$.
\end{cor}
\begin{proof}
This follows from \aref{thm:Comparison} and \cite[Theorem~2.27]{V:cd}.
\end{proof}
\begin{remark}
 All of the statements above hold as well for $\Sm^{G}_{S}$.
\end{remark}

The original definition of the equivariant Nisnevich topology, due to Voevodsky \cite[Section 3.1]{Deligne}, defined the covers in terms of an equivariant  splitting property. 
We show that the definition given here in terms of a $cd$-structure agrees with that definition. See \aref{section:finite} for another characterization when the group is finite.

\begin{defn}\label{defn:splitting}
An equivariant morphism $Y\to X$ in $\Sch^G_S$ has an {\sl equivariant  splitting sequence} 
if there is a filtration of $X$ by invariant closed subschemes
\begin{equation}\label{eqn:split*-0}
\emptyset = X_{n+1} \subsetneq X_{n} \subseteq \cdots \subseteq X_0 = X, 
\end{equation}
such that for each $j$, the map
\[
\left(X_j \setminus X_{j+1}\right) \times_X Y \to X_j \setminus X_{j+1}
\]
has an equivariant section. We say that this splitting sequence has length $n$.
\end{defn}

\begin{prop}\label{prop:Nisne-split}
An equivariant \'etale morphism $Y \xrightarrow{f} X$ in $\Sch^G_S$ is an equivariant Nisnevich cover if and only if it has an equivariant splitting sequence.
\end{prop}
\begin{proof}

Suppose that $f:Y\to X$ is an equivariant Nisnevich cover.  Note that there is a dense invariant open subscheme $U\subseteq X$ on which $f$ has a splitting. Indeed, this is true by definition for covers coming from distinguished squares and this property is preserved by pullbacks and by compositions. An invariant open subscheme has an invariant closed complement. Restricting to an invariant closed complement of $U$ and repeating the argument, we construct an equivariant splitting sequence, which must stop at a finite stage because $X$ is Noetherian.

For the converse, we proceed by induction on the length of a splitting sequence. The case of length zero is immediate.  Suppose that $f$ has an equivariant splitting sequence of length $n$. The restriction of $f$ to $X_{n}\times_{X}Y\to X_{n}$ has an equivariant section $s$. Since $s$ is equivariant and \'etale, $s(X_{n})\subseteq X_{n}\times_{X} Y$ is an invariant open. Let $D$ be an invariant closed complement. 
Consider the map $\widetilde{Y}:=Y\setminus D\to X$. 
Then $\{\widetilde{Y}\to X,\,X-X_{n}\}$ forms an equivariant distinguished covering of $X$. 
The pullback of $f:Y\to X$ along $X-X_{n}$ has an equivariant splitting sequence of length less than $n$ and so by induction is an equivariant Nisnevich cover. Similarly the pullback of $f$ along $\widetilde{Y}\to X$ equivariantly splits and is thus also an equivariant Nisnevich cover. It follows that $f$ itself is an equivariant Nisnevich covering.
\end{proof}

\begin{exm}\label{exm:Ison-H-1}
Let $C_2=\<\sigma\>$ be the cyclic group of order two. Let
$X$ denote the smooth $C_2$-scheme over $\Spec(\R)$ defined by $\Spec(\C)$ equipped with complex conjugation. Let $Y= \Spec(\C)\coprod \Spec(\C)$ with $C_2$-action given by switching the factors. Let $f:Y\to X$ be given by the identity on one factor and 
complex conjugation on the other factor.
Then  $f$ is a Nisnevich cover after forgetting the $C_2$-action but is not locally equivariantly split and so is not an equivariant Nisnevich cover.
\end{exm}

\subsection{Finite groups}\label{section:finite}
In this section we focus on the case of a finite constant group scheme. 
Throughout this subsection, $G$ is a finite group (in the category of sets). The associated group scheme over $S$ given by $\coprod_{G} S$ is denoted as well by $G$. Given a subgroup $H\subseteq G$ and an $H$-scheme $Z$ we write $G\times^{H} Z:= (G\times Z)/H$.

If $X$ is a $G$-scheme and $x\in X$ is a point, the {\sl set-theoretic} stabilizer of $x$ is the subgroup  $S_x\subseteq G$ defined by 
$S_{x}= \{g\in G \mid g\cd x =x\}$. 
The orbit of $x$ is $G\cd x: = G\times^{S_{x}} x$, which has underlying set  $\{g\cd x \mid g\in G\}$.

\begin{prop}\label{prop:nischar}
An equivariant \'etale map $f:Y\to X$
 in $\Sch^G_S$ is an equivariant Nisnevich cover if and only if 
for any point $x \in X$ there is a point $y \in Y$ such that $f(y)= x$ and $f$ induces isomorphisms $k(x)\iso k(y)$ and $S_{y}\iso S_{x}$.
\end{prop}
\begin{proof}
This is proved in \cite[Proposition 3.5]{HVO} when $S=\Spec(k)$ is a field. The same proof applies for a general base $S$. For convenience, we repeat the proof here.
First, by \aref{prop:Nisne-split}, if $f$ is an equivariant Nisnevich cover then it has an equivariant splitting sequence as in (\ref{eqn:split*-0}). Then 
$x\in  X_{j}-X_{j+1}$ for some $j$. Let $s$ be a section of $f$ over $X_{j}-X_{j+1}$ and let $y = s(x)$. One immediately verifies that $f$ induces an isomorphism $k(x) \iso k(y)$ and $S_{y}\iso S_{x}$.

For the other direction, again by \aref{prop:Nisne-split}, it suffices to show that $f$ has a splitting sequence. By Noetherian induction, it suffices to show that if for each generic point $\eta\in X$ there is $\eta'\in Y$ so that $f$ induces $k(\eta)\iso k(\eta')$ and $S_{\eta}\iso S_{\eta'}$ then there is an equivariant dense open $U\subset X$ such that $Y\times_{X}U\to U$ has an equivariant splitting. To show this it suffices to assume that $X$ is equivariantly irreducible. Let $\eta\in X$ be a generic point. Then there is an $\eta'\in Y$ such that 
$f:\eta'\iso \eta$ and $S_{\eta'}\iso S_{\eta}$. This implies that 
$G\cd \eta'\to G\cd \eta$ is an equivariant isomorphism. 
We have that $G\cd \eta'\iso \cap W'$  
(resp.~$G\cd \eta$) is the intersection over all invariant opens $W'$ in $Y$ containing $\eta'$ (resp.~ all invariant opens in $X$) and so there is some invariant open $W'\subseteq Y$ such that $W'\to f(W')$ is an equivariant isomorphism. Setting $U=f(W')$ we obtain our equivariant splitting. 
\end{proof}

\begin{remark}
The proof of the previous proposition shows as well the following useful fact: the version of the equivariant Nisnevich topology defined using ``infinite'' covers yields a site which is equivalent to the one we have defined here. 
\end{remark}

\begin{cor}
\label{cor:finitesubcovering}
Let $\{f_{i}:Y_{i}\to X\}_{i\in I}$ be a collection of equivariant \'etale maps in $\Sch_{S}^{G}$ such that for every $x\in X$ there is an index $i=i(x)\in I$ and a point $y\in Y_{i}$ such that $f_{i}$ induces isomorphisms $k(x)\iso k(y)$ and $S_{y}\iso S_{x}$.  
Then there is a finite sub-collection $\{Y_{i_{j}}\to X\}_{j=1}^{n}$ which is an equivariant Nisnevich cover.
\end{cor}

We now turn our attention to the points of the equivariant Nisnevich topology.
Recall that a {\sl point} $x$ on a Grothendieck site $\sC$ is a functor $x^*\colon\Shv(\sC)\to\Sets$ from sheaves on $\sC$ to sets which 
commutes with all small colimits and finite limits. 
Such a functor has a right adjoint $x_*\colon\Sets\to\Shv(\sC)$ by Freyd's adjoint functor theorem. 
An explicit description of the points of the equivariant Nisnevich topology is provided in \cite{Deligne} (in the case of quasi-projective $G$-schemes). We proceed in a somewhat different fashion to describe the points. 
If the orbit $G\cd x$ of a point $x\in X$  is contained in an affine neighborhood, then $\sO_{X,Gx}$ is a semilocal ring. In this case  
we let $\sO^h_{X, Gx}$ denote the henselization of the semilocal ring $\sO_{X, Gx}$ along the ideal defining the scheme $G\cd x$. Note that the semilocal ring $\sO_{X,Gx}^{h}$ has a $G$-action coming from the action on $X$ because henselization is functorial. In general, $G\cd x$ is not contained in an invariant affine Zariski neighborhood. 
Instead, we consider the category $N_{G}(G\cd x)$  of affine equivariant Nisnevich neighborhoods 
of $G\cd x$. An {\sl equivariant Nisnevich neighborhood} of $G\cd x$ is  an equivariant \'etale map $f:Y\to X$ and an equivariant map $s:G\cd x\to Y$ such that the triangle commutes
$$
\xymatrix@C1pc{
 & Y \ar[d]^{f} \\
 G\cd x \ar[r]\ar[ur]^{s} & X .
}
$$
The category $N_{G}(G\cd x)$ is filtering. We note that it is also nonempty. 

\begin{lemma}\label{lem:affinenbd}
Any orbit $G\cd x$ is contained in an affine equivariant Nisnevich neighborhood.
\end{lemma}
\begin{proof}
The point $x$ is contained in an affine $S_x$-invariant neighborhood $U$. Note that the map $G\times^{S_{x}}U\to X$
is an equivariant Nisnevich neighborhood of $G\cd x$.
\end{proof}

We define
\begin{equation}\label{eqn:Ghensel}
X^{h}_{Gx}:= \lim _{U\in N_{G}(G\cd x)} U.
\end{equation}
The transition maps of this filtered limit are affine and so this limit exists as a scheme. The $G$-action on $X$ induces one on $X^{h}_{Gx}$.

If $U\to X$ and $g \in G$, 
the {\sl translate} of $U$ by $g$ is the scheme $g(U)$ defined by the cartesian square
\begin{equation*}
\xymatrix@C1pc{
g(U) \ar[r] \ar[d]_{g(f)} & U \ar[d]^{f} \\
X \ar[r]_{\tau_{g^{-1}}} & X,}
\end{equation*}
where $\tau_{g^{-1}}: X \to X$ is the automorphism defined by $g^{-1}$ via the $G$-action on $X$.  
When $G = \{e = g_0, \cdots , g_n\}$ is finite we can iteratively form the fiber product 
$$
U_G:= U \times_X \ g_1(U)
\times_{X} \cdots \times_{X} \ g_n(U),
$$
using the maps $g_i(f):g_i(U) \rightarrow X$. 
Now suppose that $Z \subseteq X$ is a $G$-invariant subset 
and $U\to X$ is a Nisnevich neighborhood of $Z$. It is straightforward to check that $U_G\to X$ is an equivariant Nisnevich neighborhood of 
$Z$ and there is a factorization
$(U_G, Z) \to (U, Z) \to (X, Z)$. 
One now readily sees that when $G\cd x$ is contained in an invariant affine neighborhood then $X^h_{Gx}\cong \Spec(\sO^{h}_{X,Gx})$. More generally we see that 
$$
X^{h}_{Gx} \cong G\times^{S_{x}}\Spec(\sO^{h}_{X,x}).
$$

Let $X$ be a smooth $G$-scheme over $S$. For any  $x\in X$, we define a functor $x^*: \Sm^G_S \to\Sets$ by setting 
$\underline{x}(U) := \Hom(X^h_{Gx}, U)$,
where the morphism set is taken in the category of all $S$-schemes with $G$-action (not necessarily of finite type). 
We extend this to a functor $x^*: \Shv(\Sm^G_{S/{\rm Nis}})\to \Sets$ via
the left Kan extension of $x^*$ along the Yoneda embedding. Explicitly, 
$x^*F = \colim_{U\in  N_{G}(Gx)}\mcal{F}(U)$.
In an entirely analogous fashion, we define a point $x^*$ on 
$\Sch^G_{S/{\rm Nis}}$ for any point $x$ of a $G$-scheme $X$ over $S$.
We usually write $F(X^h_{Gx}) = x^*F$.

\begin{prop}\label{prop:Point-Cons}
Let $G$ be a finite group. The collection 
$\{ x^* \, \vert \, X \in \Sm^G_S, x \in X\}$ 
(resp.~ $\{x^* \, \vert \, X \in \Sch^G_S, x \in X\}$) 
forms a conservative family of points on the site $\Sm^G_{S/{\rm Nis}}$ (resp.~ on $\Sch^G_{S/{\rm Nis}}$).
\end{prop} 
\begin{proof}
We treat the case $\Sm^{G}_{S}$, the case of $\Sch^{G}_{S}$ being verbatim. By \cite[Proposition~6.5.a]{SGA4Tome1} it is enough to show that for $U \in \Sm^G_S$ and $\{f_i: U_i \to U\}_{i \in I}$ a family of $G$-equivariant 
maps such that $\{x^*(U_i) \to x^*(U)\}_{i \in I}$ is surjective for all $X \in \Sm^G_S$ and all $x \in X$, 
the map $\{f_i\}$ is dominated by an equivariant Nisnevich cover of $U$.
Suppose that $\{x^*(U_i) \to x^*(U)\}_{i \in I}$ is a surjective family for all pairs $(X, G\cd x)$. 
For $u \in U$ and the induced $G$-equivariant map $v: (U^h_{Gu}, G\cd u) \to 
(U, G\cd u)$, there exists,
by our assumption, 
an index $i \in I$ and a $G$-equivariant factorization
\begin{equation*}\label{eqn:Point*-1}
\xymatrix@C1.2pc{
& U_i \ar[d]^{f_i} \\
U^h_{Gu} \ar[ur]^{w}\ar[r]_{v} & U.}
\end{equation*}
Notice that $w$ is an isomorphism when restricted to $G\cd u$, and hence it gives a section of $f_i$ over $G\cd u$. 
Since $(U^h_{Gu}, G\cd u)$ is the filtered limit of equivariant Nisnevich neighborhoods of $G\cd u$ and $f$ is a $G$-equivariant finite type morphism, 
there is an equivariant Nisnevich neighborhood $(U'_i, G\cd u)$ and a $G$-equivariant factorization $(U'_i,G\cd u)\xrightarrow{w}(U_i,G\cd u)\xrightarrow{f_i} (U,G\cd u)$. 
The point $u \in U$ was chosen arbitrarily,
so we deduce the desired domination of $\{f_i\}$.
\end{proof}

\section{Local model structures}\label{section:Models}
Let $\sS$ denote the category of simplicial sets with internal hom $\sS(-,-)$ defined in \cite[I.5]{GoerssJardine}. 
Similarly, 
let $\sS_{\bullet}$ denote the category of pointed simplicial sets with internal hom $\sS_{\bullet}(-,-)$. 
\begin{definition}
A \textsl{motivic $G$-space} $\sX$ is a presheaf 
$\sX:(\Sm^G_S)^{op} \to \sS$ of simplicial sets. A \textsl{pointed motivic $G$-space} is a presheaf $(\Sm^G_S)^{op} \to \sS_{\bullet}$ of pointed simplicial sets.
\end{definition}
By the finite type condition the category $\Sm^G_S$ is essentially small, 
i.e.,
it is locally small with a small set of isomorphism classes of objects.
Let $\sM^G(S)$ (resp.~$\sM^G_{\bullet}(S)$) denote the category of motivic (resp.~pointed motivic) $G$-spaces. 
We identify $\sS$ with the full subcategory of $\sM^G(S)$ comprised of constant motivic $G$-spaces. 
The Yoneda lemma yields a fully faithful embedding of $\Sm^G_S$ into $\sM^G(S)$ associating an object $X$ of $\Sm^G_S$ with the representable motivic 
$G$-space $h_{X}(-):= \Hom_{\Sm^G_S}(-, X)$ which takes values in discrete simplicial sets. 
We will usually make no notational distinction between $X$ and $h_X$. 
It follows from \aref{cor:sub-can} that $h_X$ is a sheaf in the equivariant Nisnevich topology.  
A pointed motivic $G$-space is a motivic $G$-space $\sX$ together with a map $pt = h_S \to \sX$. 
For $X \in \Sm^G_S$, the symbol $X_{+}$ denotes the pointed motivic $G$-scheme $(X \coprod pt, pt)$. 
We note the following useful fact about $\sM^G_{\bullet}(S)$.

\begin{lem}\label{lem:elementary}
The category $\sM^G_{\bullet}(S)$ is both a closed symmetric monoidal category and a locally finitely presented bicomplete $\sS_{\bullet}$-category. 
In particular, filtered colimits commute with finite limits.
\end{lem}

The tensor product in $\sM^G_{\bullet}(S)$ is defined by taking the pointwise, 
or schemewise, 
smash product $(\sX \wedge \sY)(U) = \sX(U) \wedge \sY(U)$. 
With this definition,
$pt = S$ is the unit of the product.
Limits and colimits in $\sM^G_{\bullet}(S)$ are also defined pointwise. 
The functor ${\rm Ev}_U$ evaluating motivic $G$-spaces at a fixed $G$-scheme $U$ is strict symmetric monoidal, 
preserves limits and colimits, 
and there is an adjunction:
\begin{equation}\label{eqn:adjunc}
\xymatrix{
{\rm Fr}_U
\colon
\sS_{\bullet}
\ar@<+.7ex>[r] &
\ar@<+.7ex>[l]
\sM^G_{\bullet}(S)
\colon 
{\rm Ev}_U.
}
\end{equation}
The left adjoint ${\rm Fr}_U$, 
defined by ${\rm Fr}_U(K) = U_{+} \wedge K$, 
is lax symmetric monoidal for any $G$-scheme and strict symmetric monoidal when $U = pt$. 
For $\sX \in \sM^G_{\bullet}(S)$ and $K \in \sS_{\bullet}$,
we define $\sX \wedge K$ and $\sX^K$ by sending $U$ to $\sX(U) \wedge K$ and $\sS_{\bullet}(K, \sX(U))$, 
respectively.  

The $\sS_{\bullet}$-enrichment of pointed motivic $G$-spaces is given degreewise by the pointed simplicial set
\begin{equation}\label{eqn:simp-str}
{\sS(\sX, \sY)}_{n} 
= 
\Hom_{\sM^G_{\bullet}(S)}(\sX \wedge \Delta[n]_{+}, \sY).
\end{equation}
The internal hom in $\sM^G_{\bullet}(S)$ is defined pointwise as $\sHom(\sX, \sY)(U) = \sS(\sX \wedge U_{+}, \sY)$. 

A pointed motivic $G$-space $\sX$ is finitely presentable if $\Hom_{\sM^G_{\bullet}(S)}(\sX, -)$ commutes with filtered colimits. 
Using the natural isomorphism $\sHom(U_{+} \wedge K, \sX) \simeq \sX(U \times -)^K$,
one deduces that $\sX$ is finitely presentable if and only if $\sS(\sX, -)$ commutes with filtered colimits. 
The pointed finite simplicial sets and the $G$-schemes form the building blocks for $\sM^G_{\bullet}(S)$ in the following sense
(see \cite[5.2.2b, 5.2.5]{handbook2}):

\begin{lem}\label{lem:finitecolimit}
Every pointed motivic $G$-space is a filtered colimit of finite colimits of pointed motivic $G$-spaces of the form $(U \times \Delta[n])_{+}$,
where $U \in \Sm^G_S$ and $\Delta[n]$ is the standard $n$-simplex for $n \ge 0$.
The motivic $G$-spaces $(U \times \Delta[n])_{+}$ are finitely presented.
The finitely presented motivic $G$-spaces are closed under retracts,
finite colimits and smash products.
\end{lem}  

In the above we described the monoidal structure on pointed motivic $G$-spaces. 
This story works verbatim for motivic $G$-spaces $\sM^G(S)$ by replacing the smash product with the product $\sX \times \sY$.

\subsection{Global model structures}\label{subsection:SMS}
We recall the standard global model structures which we later localize to obtain motivic model structures on motivic $G$-spaces.
We refer the reader to \cite{Hirsc} for standard notions related to model structures. 
Recall that a model structure on $\sM^G(S)$ is {\sl simplicial} if the simplicial structure interacts with cofibrations, 
fibrations and weak equivalences as follows: 
If $i :\sX \to \sY$ is a cofibration and $p : \sZ \to \sW$ a fibration in $\sM^G(S)$, 
then the map of simplicial sets
\[
\sS(\sY, \sZ) \xrightarrow{(i^*, p_*)}
\sS(\sX,\sZ) {\underset{\sS(\sX,\sW)}\times} \
\sS(\sY, \sW)
\]
is a Kan fibration, 
which is a weak equivalence if either $i$ or $p$ is a weak equivalence.

We  say that a map $f : \sX \to \sY$ of motivic $G$-spaces is a 
{\sl schemewise weak equivalence} (resp.~{\sl schemewise fibration}) 
if the map of simplicial sets $\sX(X) \to \sY(X)$ is a weak equivalence (resp.~Kan fibration) of simplicial sets for every $X \in \Sm^G_S$. A schemewise fibration will be more frequently called a \textsl{projective fibration}.
A \textsl{projective cofibration} is a map
$f$ which has the left lifting property with respect to all maps which are schemewise fibrations and weak equivalences. 
Using \cite[Theorems~11.6.1, 11.7.3, 13.1.14, Proposition~12.1.5]{Hirsc}, 
one deduces the existence and standard properties of the {\sl projective model structure}.

\begin{thm}[Projective model structure]\label{thm:PMS-Psh} 
The schemewise weak equivalences, projective fibrations,
and projective cofibrations form a cellular, combinatorial and simplicial model structure on $\sM^G(S)$ with respect to the $\sS$-enrichment 
in~\eqref{eqn:simp-str}.

The set of generating cofibrations
\[
I^{\rm sch}_{\rm proj}(\Sm^G_S) 
= 
\{U \times (\partial \Delta^n \subset\Delta^n)\}_{n \ge 0, U \in \Sm^G_S}
\]
and trivial cofibrations
\[
J^{\rm sch}_{\rm proj}(\Sm^G_S) 
= 
\{U \times (\Lambda^n_i \subset\Delta^n)\}_{n \ge 1, 0 \le i \le n, U \in \Sm^G_S}
\]
are induced from the corresponding maps in $\sS$. 
The domains and codomains of the maps in these generating sets are finitely presented.
The projective model structure is proper. 
For every $U \in \Sm^G_S$ the pair $({\rm Fr}_U, {\rm Ev}_U)$ forms a Quillen pair.
\end{thm}

An \textsl{injective cofibration} is a schemewise cofibration.
Let $\kappa$ be the first cardinal number greater than the cardinality of the set of maps in the category of presheaves on $\Sm^G_S$. 
If $\omega$ denotes, 
as usual, 
the cardinal of continuum, 
we define $\gamma$ as ${\kappa \omega}^{\kappa \omega}$. 
Now let $I^{\rm sch, \kappa}_{\rm inj}(\Sm^G_S)$ be the set of maps $\sX \to \sY$ such that $\sX(U) \to \sY(U)$ is a cofibration of simplicial sets 
of cardinality less than $\kappa$ for every $U \in \Sm^G_S$. 
Likewise, 
we define $J^{\rm sch, \gamma}_{\rm inj}(\Sm^G_S)$ for schemewise trivial cofibrations of simplicial sets bounded by $\gamma$. 
The {\sl injective model structure} on $\sM^G(S)$, defined in \cite{Heller}, has the following properties, see e.g., \cite[Theorem 1.4]{Hornbostel}, \cite[Theorem 2.16]{Barwick}, or \cite[Proposition A.3.3.2]{Lurie}.

\begin{thm}[Injective model structure]\label{thm:LMS-Psh} 
The  schemewise weak equivalences, injective cofibrations, and injective fibrations form a  cellular, combinatorial and simplicial model structure on $\sM^G(S)$ with respect to the $\sS$-enrichment 
in~\eqref{eqn:simp-str}.
The cofibrations and trivial cofibrations are generated by $I^{\rm sch, \kappa}_{\rm inj}(\Sm^G_S)$ and $J^{\rm sch, \gamma}_{\rm inj}(\Sm^G_S)$,
respectively. 
\end{thm}
 
The third model structure we consider is the flasque model structure \cite{Isaksen:flasque}. 
For $U \in \Sm^G_S$, 
consider a finite set of equivariant monomorphisms $V_I = \{V_i \to U\}_{i \in I}$. 
The categorical union $\cup_{i\in I} V_i$ is the coequalizer of the diagram in $\sM^G(S)$
 
$$
\coprod_{i,j\in I}  V_i \times_{U} V_j
\twoarrow
\coprod_{i\in I}  V_i.
$$

Let $i_I$ denote the induced monomorphism ${\cup_{i\in I}} V_i \to U$. 
Note that $\0 \to U$ arises in this way. 
The pushout product of maps of $i_I$ and a map between simplicial sets exists in $\sM^G(S)$. 
In particular, 
we are entitled to form the sets
\[
I^{\rm sch}_{\rm fl}(\Sm^G_S) 
= 
\{i_I  \ \square \ (\partial \Delta^n \subset \Delta^n)\}_{I, n \ge 0}
\]
and
\[
J^{\rm sch}_{\rm fl}(\Sm^G_S) 
= 
\{i_I \ \square \ (\Lambda^n_i \subset \Delta^n)\}_{I, n \ge 1, 0 \le i \le n}.
\]

A map between motivic $G$-spaces is a flasque fibration if it has the right lifting property with respect to $J^{\rm sch}_{\rm fl}(\Sm^G_S)$.
Moreover,
a flasque cofibration is a map having the left lifting property with respect to every trivial flasque fibration. The flasque model structure satisfies the following properties, see \cite{Isaksen:flasque}.

\begin{thm}[Flasque model structure]\label{thm:PMS-Flasque}
The schemewise weak equivalences, 
flasque cofibrations and fibrations form 
form a cellular, combinatorial and simplicial model structure on $\sM^G(S)$ with respect to the $\sS$-enrichment 
in~\eqref{eqn:simp-str}.
The flasque cofibrations and fibrations are generated by $I^{\rm sch}_{\rm fl}(\Sm^G_S)$ and $J^{\rm sch}_{\rm fl}(\Sm^G_S)$, 
respectively.
\end{thm}

Finally we note the following.
\begin{theorem}\label{thm:globaleq}
 The identity functor is a left Quillen equivalence from the global projective to the global flasque and a left Quillen equivalence from the global flasque to the global injective model structures.
\end{theorem}

\subsection{Local model structures}\label{subsection:LPMS}
Next we introduce the local model structures, which take into account the equivariant Nisnevich topology. Local equivalences for presheaves of simplicial sets
on a Grothendieck site are defined via sheaves of homotopy groups, 
see \cite{Jardine:spre}. As shown in \cite{DHI} this approach is equivalent to a Bousfield localization at the class of hypercovers. Furthermore, by \cite{V:cd}, when the topology is defined via a $cd$-structure it suffices to localize with respect to the distinguished squares.

Let  $\sM$ be a simplicial model category and $\Sigma$ a class of morphisms. 
Recall from \cite[Chapter~3]{Hirsc} that 
an object $\sZ$ of $\sM$ is called \textsl{$\Sigma$-local} 
if it is fibrant and for every element $f : \sX \to \sY$ in $\Sigma$, 
the induced map of simplicial function complexes 
$\sS(\sY^{\cof}, \sZ) \to \sS(\sX^{\cof}, \sZ)$ is a weak equivalence (see \cite[Definitions~3.1.4, 17.1.1]{Hirsc}), where $(-)^{\cof}$ denotes cofibrant replacement.
Moreover, 
a map $f : \sX \to \sY$ in $\sM$ is a $\Sigma$-local equivalence if for every $\Sigma$-local object $\sZ$, 
the induced map of homotopy function complexes 
$\sS(\sY^{\cof}, \sZ) \to \sS(\sX^{\cof}, \sZ)$ is a weak equivalence. 
Clearly every element of $\Sigma$ is a $\Sigma$-local equivalence.

The {\sl left Bousfield localization} of $\sM$ with respect to the class $\Sigma$ of morphisms is a model category structure $L_{\Sigma}\sM$ 
on the underlying category $\sM$ with  
\begin{enumerate}
\item
weak equivalences the $\Sigma$-local equivalences of $\sM$, 
\item
cofibrations are the same as the cofibrations of $\sM$, and
\item
fibrations the maps having the right lifting property with respect to the cofibrations that are simultaneously $\Sigma$-local equivalences.
\end{enumerate}

If $\Sigma$ is a set of maps in a left proper and cellular \
(or combinatorial) model category $\sM$, then by \cite[Theorem 4.1.1]{Hirsc} (or J. Smith's theorem \cite[Theorem 4.7]{Barwick}) 
the left Bousfield localization $L_{\Sigma}\sM$ exists. 
Moreover, 
\begin{enumerate}
 \item[(4)] the fibrant objects of $L_{\Sigma}\sM$ are the $\Sigma$-local objects.
\end{enumerate}

Now we localize the global model structures. 
For a distinguished equivariant Nisnevich square $Q$ as in ~\eqref{eqn:cd-square} write $Q^{hp}$ for the homotopy pushout in the global projective model structure. 
There is a canonical map $Q^{hp} \to X$ and we set 
\begin{equation*}
\Sigma^{hp}_{\rm Nis} = \{Q^{hp} \to X\}_Q \cup \{\emptyset \to h_{\emptyset}\}.
\end{equation*}

Here $\emptyset$ is the initial motivic $G$-space and $h_{\emptyset}$ is the motivic $G$-space represented by the empty $G$-scheme.

\begin{definition}
 The {\sl local projective} (resp.~ {\sl local flasque}, resp.~ {\sl local injective} model structure on $\sM^G(S)$ is the left Bousfield 
 localization at $\Sigma^{hp}_{\rm Nis}$ of the global projective (resp.~ global flasque, resp.~ global injective) model structure. 
\end{definition}

In the case of the local flasque and the local injective model structures, instead
of using the homotopy pushout $Q^{hp}$ of a distinguished square, one could use instead
the pushout (computed in $\sM^G(S)$). The point is that a monomorphism of smooth
$G$-schemes is a cofibration in these model structures, and so the categorical pushout
is weakly equivalent to a homotopy pushout in these model structures. However, to
keep the treatment uniform, we use the homotopy pushout in all three cases.

Recall \cite{Jardine:spre} that a map $f:\sX\to \sY$ is a {\sl local equivalence} if it induces an isomorphisms  $(\pi_{0}\sX)_{GNis}\iso (\pi_{0}\sY)_{GNis}$ and $(\pi_{n}(\sX|_{U},x))_{GNis}\iso (\pi_{n}(\sY|_{U}, f(x)))_{GNis}$ of sheaves for every $U$ in $\Sm^{G}_{S}$ and basepoints $x\in \sX(U)$ and $n\geq 0$.

\begin{thm}\label{thm:LFMS-Psh}
The local projective, local flasque, and local injective model structures on $\sM^{G}(S)$ are cellular, combinatorial, proper, and simplicial. 
The identity functors from the local projective model structure to the local flasque and local injective model structures are left Quillen equivalences. Moreover, the weak equivalences in all three model structures agree and are exactly the local weak equivalences.
\end{thm}
\begin{proof}
The schemewise model structures are cellular, combinatorial and left proper ones,
and hence the Bousfield localizations defining these model structures exist, are cellular, combinatorial, left proper, and simplicial (see \cite[Theorem~4.1.1]{Hirsc} or \cite[Theorem 4.7]{Barwick}).
The identity functor induces a left Quillen equivalence from the local projective to the local flasque to the local injective model structures because it does so on global model structures (\aref{thm:globaleq}). 

Right properness of the local projective model structure
follows as in \cite[Theorem~1.5]{Blander}.
The other model structure are also right proper because
 a local injective fibration or a local flasque fibration, 
is also a local projective fibration.

By \cite[Theorem~3.8]{V:cd}, the weak equivalences in the local projective model structure are exactly the local equivalences. It is straightforward to verify that local flasque and local injective equivalences are also exactly the local equivalences (see e.g., \cite[Theorem 4.3]{Isaksen:flasque}).
\end{proof}

A presheaf of simplicial sets 
$\sF$ on a Grothendieck site $\mcal{C}$ is said to satisfy \textsl{$\tau$-descent} if every hypercover $U_{\bullet}\to U$ induces a weak equivalence
$\sF(U) \simeq \holim_{\Delta}\sF(U_{n})$. Equivalently every fibrant replacement $\sF \to \wh{\sF}$, in the local model structure, is an objectwise equivalence. 
An important feature of topologies defined via $cd$-structures is that descent is equivalent to a substantially simpler excision property.

\begin{definition}
 A motivic $G$-space $\sX$ is said to be \textsl{equivariant Nisnevich excisive} provided
 \begin{enumerate}
 \item[(i)] $\sX(\emptyset)$ is contractible, and 
 \item[(ii)] the square is homotopy cartesian
 \begin{equation*}
\xymatrix{
\sX(X)  \ar[r] \ar[d] & \sX(A) \ar[d] \\
\sX(Y) \ar[r] & X(B)} 
\end{equation*}
for every distinguished equivariant Nisnevich square~\eqref{eqn:cd-square} in $\Sm_{S}^{G}$.
\end{enumerate}
\end{definition}

\begin{remark}
A motivic $G$-space $\sX$ is locally projective (resp.~ flasque, resp.~ injective)
fibrant if and only if it is globally projective (resp.~ flasque, resp.~ injective) fibrant and is equivariant Nisnevich excisive. 
\end{remark}

\begin{prop}\label{prop:Flasq-fib}
Let $\sX$ be a motivic $G$-space. The following are equivalent.
\begin{enumerate}
\item[(i)] $\sX$ is equivariant Nisnevich excisive,
\item[(ii)] any local fibrant replacement $\sX \to \wh{\sX}$  is a schemewise weak equivalence. 
\end{enumerate}
\end{prop}
\begin{proof}
A fibrant replacement in the local injective or flasque model structure is also a local
projective fibrant replacement and so it suffices to consider this case. In this case the result follows from \cite[Proposition 3.8, Lemma 3.5]{V:cd}.

\end{proof}

\section{Motivic model structures}\label{section:UHC}
In this section we introduce the unstable homotopy category of motivic $G$-spaces, which is defined as the $\A^1$-localization of the local model structure.

\subsection{Unpointed motivic spaces}
\label{section:A-Local}
In what follows we consider $\A^1_S$ with trivial $G$-action and for simplicity we usually write $\A^1$, omitting mention of the base scheme.

\begin{defn}\label{defn:A-1model-str}
The motivic projective (resp.~injective, flasque) model structure on $\sM^G(S)$ is the left Bousfield localization of the  local 
projective (resp.~injective, flasque) model structure with respect to the set of projection maps 
\[
\{X \times \A^1 \xrightarrow{p_X} X \,\vert\, \ X \in \Sm^G_S\}.
\] 

The weak equivalences in each of the motivic projective (resp.~injective, flasque) model structures coincide with each other. A weak equivalence in any of these model structures will be simply called a {\sl motivic weak equivalence}. 
\end{defn}

\begin{remark}
This definition is made so that the affine line becomes contractible. In fact, it forces {\sl all} representations to become contractible and more generally equivariant vector bundle projections become equivalences, 
see \aref{prop:Hom-inv}. In particular allowing the affine line to be replaced with nontrivial linear actions in the above definition does not lead to a more general model structure.  
\end{remark}

\begin{theorem}
 The motivic projective (resp.~injective. flasque) model structure on $\sM^G(S)$ is a proper, cellular and combinatorial, and simplicial model structure. Moreover the  identity functors from the motivic projective to the motivic flasque and injective model 
structures are left Quillen equivalences. 
 \end{theorem}
\begin{proof}
Combining \aref{thm:LFMS-Psh}, \cite[Theorem~4.1.1]{Hirsc}, and \cite[Theorem 4.7]{Barwick}) we have that the motivic projective, injective and flasque model structures 
are left proper, cellular, combinatorial and simplicial. 
Moreover, 
right properness of the motivic model structures follow as in 
\cite[Lemma~3.1]{Blander}. The last statement follows from the fact that the identity functor is  a Quillen equivalence between the global projective, flasque, and injective model structures. 
\end{proof}

The unstable (unpointed) equivariant motivic homotopy category $\Ho^G(S)$ is the homotopy category associated to the motivic model structure on $\sM^G(S)$. 
The following description of fibrant objects follows immediately from the definition of the motivic model structure and standard properties of Bousfield localization (see the beginning of \aref{subsection:LPMS}).

\begin{lem}\label{lem:A1-compare}
A motivic $G$-space $\sX$ is fibrant in the motivic projective (resp.~injective, flasque) model structure if and only if
\begin{enumerate}
 \item[(i)] $\sX$ is
local projective (resp. injective, flasque) fibrant, and 

\item[(ii)] 
 $\sX(X)\to \sX(X\times\A^1)$ is a weak equivalence for all $X$ in $\Sm^{G}_{S}$.
\end{enumerate}
\end{lem}

A motivic $G$-space $\sX$ is said to be \textsl{$\A^{1}$-invariant} provided $\sX(X)\to \sX(X\times\A^1)$ is a weak equivalence for all $X$ in $\Sm^{G}_{S}$.

\begin{thm}\label{thm:A1-flasq}
Let $\sX$ be a motivic $G$-space. The following  are equivalent.
\begin{enumerate}
\item[(i)] $\sX$ is equivariant Nisnevich excisive and is $\A^{1}$-invariant.
\item[(ii)] Any fibrant replacement $\sX\to \mcal{Q}\sX$ in the motivic projective (resp. flasque, injective)  model structure is a schemewise 
weak equivalence.
\end{enumerate}
 Moreover, if $f: \sX \to \sY$ is a map between motivic $G$-spaces which satisfy these equivalent conditions then $f$ is a motivic weak equivalence if and only if it is a schemewise weak equivalence.
\end{thm}
\begin{proof}
That (ii) implies (i) follows from \aref{prop:Flasq-fib} and \aref{lem:A1-compare}.

For the converse we treat the case of the injective model structure explicitly, the other cases are the same. 
Suppose that $\sX$ satisfies (i) and let $f: \sX \to \mcal{Q}\sX$ be a motivic injective fibrant replacement. 
By \aref{prop:Flasq-fib},
it is enough to show that $f$ is a local injective fibrant replacement. 

We factor $f$ as a composition $\sX \xrightarrow{g} \sX' \xrightarrow{f'} \mcal{Q}{\sX}$, where $g$ is a local trivial cofibration 
(in particular, a motivic trivial cofibration) and $f'$ is a local injective fibration. 
It follows from the 2-out-of-3 axiom that $f'$ is a motivic weak equivalence. 
We need to show that $f'$ is a local weak equivalence.

Since $\mcal{Q}{\sX}$ is local injective fibrant and $f'$ is a local injective fibration, 
it follows that $\sX'$ is also local injective fibrant. 
In particular, 
$g$ is a local injective fibrant replacement for $\sX$. 
We conclude from \aref{prop:Flasq-fib} that $g$ is a schemewise weak equivalence.
Note that since $\sX$ satisfies condition (i) so does $\sX'$. By  \aref{lem:A1-compare} we conclude that $\sX'$ is motivic injective fibrant.
Since $f'$ is a motivic weak equivalence between fibrant motivic $G$-spaces we conclude 
 from the local Whitehead theorem (see \cite[Theorem~3.2.12]{Hirsc})  that $f'$ is in fact a schemewise weak equivalence. 
This proves the first part of the theorem.

The second assertion of the theorem follows easily by considering motivic fibrant replacements.
\end{proof}

\subsection{Pointed motivic spaces} 

\label{section:UHC-P}
The category $\sM^G_{\bullet}(S)$ of pointed motivic $G$-spaces is the category whose objects  are pairs $(\sX,x)$ where $\sX$ is a motivic $G$-space and $x:pt\to \sX$ is a distinguished basepoint. Maps in this category are required to respect the basepoint. We usually omit explicit mention of the basepoint in notation when no confusion can arise.
By \aref{lem:elementary}, 
the category $\sM^G_{\bullet}(S)$ of pointed motivic $G$-spaces is a closed symmetric monoidal category with respect to the smash product and pointed internal hom. 
There is an adjoint functor pair 
$$
\xymatrix{
(-)_+:\sM^G(S) 
\ar@<+.7ex>[r] &
\ar@<+.7ex>[l]
\sM^G_{\bullet}(S)
}
$$
where $(\sX)_+ = \sX \coprod pt$ (which is pointed at the newly added disjoint point) and the right adjoint is the forgetful functor.

Since $\sM^G_{\bullet}(S)$ is the slice category $pt \downarrow \sM^G(S)$, the motivic projective (resp. flasque, injective) model structure on pointed motivic $G$-spaces follows from that on unpointed motivic $G$-spaces by \cite[Theorem~7.6.5]{Hirsc}. 

\begin{thm}\label{thm:Pntd-MS}
The category $\sM^G_{\bullet}(S)$ admits a proper, cellular, combinatorial, simplicial model structure with the property that a map 
$f : (\sX, x) \to (\sY, y)$ is a weak equivalence 
(resp.~cofibration, resp.~fibration) if and only if $f : \sX \to \sY$ is a weak equivalence (resp.~cofibration, resp.~fibration) 
in the motivic projective model structure on $\sM^{G}(S)$. The motivic flasque and motivic injective model structures on $\sM^{G}(S)$ similarly induce model structures on $\sM^{G}_{\bullet}(S)$. The identity is a Quillen equivalence between the projective, flasque, and injective model structures on $\sM^{G}_{\bullet}(S)$.
\end{thm}

The equivariant pointed motivic homotopy category 
$\Ho^G_{\bullet}(S)$ is the homotopy category associated to 
any of the equivalent motivic model structures on $\sM^{G}_{\bullet}(S)$.

\begin{prop}\label{prop:Ho-Smash}
The smash product preserves motivic weak equivalences and injective cofibrations in $\sM^G_{\bullet}(S)$. 
This  induces a symmetric closed monoidal category structure on $\Ho^G_{\bullet}(S)$.
\end{prop}
\begin{proof}
The same argument as in \cite[Lemma~2.20]{DRO} shows that smashing with any pointed motivic $G$-space preserves motivic weak equivalence.
Since the cofibrations in the motivic injective model structure are monomorphisms, 
it follows immediately that smash product preserves cofibrations.

The first assertion implies the smash product defines a structure of symmetric monoidal structure on ${\Ho}^G_{\bullet}(S)$.
We need to show that this monoidal structure is closed. For this we may use any of the equivalent model structures. We use the motivic projective model structure and it 
suffices to show that the motivic projective model structure on $\sM^G_{\bullet}(S)$ is monoidal. 
This follows from the same argument as in \cite[Corollary~2.19]{DRO}.
\end{proof}  

Recall that the simplicial circle $S^1$ is the constant presheaf ${\Delta^1}/{\partial \Delta^1}$ pointed by the image of $\partial\Delta^1$. 
As usual we write $S^n$ for $(S^1)^{{\wedge} n}$.
Smashing with the simplicial circle gives a functor
\[
\Sigma \sF = S^1\wedge \sF.
\]
Let $\Omega^1(-) = \sHom_{\bullet}(S^1, -)$ be the right adjoint of $S^1\wedge (-)$. 
\aref{prop:Ho-Smash} implies that $\left(\Sigma(-),\Omega^1(-)\right)$ is a Quillen pair of endofunctors on $\sM^G_{\bullet}(S)$. 
In particular, 
we get an adjoint pair of endofunctors

\begin{equation*}
\xymatrix{
{\bf L}\Sigma(-)
\colon
{\Ho}^G_{\bullet}(S) 
\ar@<+.7ex>[r] &
\ar@<+.7ex>[l]
{\Ho}^G_{\bullet}(S)
\colon 
{\bf R}\Omega^1(-).
}
\end{equation*}   

For  $X$ in $\Sm_{S}^{G}$, there is an adjoint pair 
${\rm Fr}_X:\sS_{\bullet} \rightleftarrows\sM^G_{\bullet}(S):{\rm Ev}_X$ 
(see \eqref{eqn:adjunc}).
Here ${\rm Fr}_X(K) := X_{+}\wedge K$ and ${\rm Ev}_X(\sF) := \sF(X)$ .

\begin{proposition}\label{prop:QA-Psh}
The functors $({\rm Fr}_X, {\rm Ev}_X)$ form a Quillen pair with respect to the motivic (resp.~local, schemewise)  projective (resp.~flasque, injective) model structures on $\sM^G_{\bullet}(S)$. 
\end{proposition}
\begin{proof}
The identity on $\sM^{G}_{\bullet}(S)$ is a left Quillen functor from the schemewise projective model structure to any of the other schemewise model structures as well any of their localizations and so we only need to consider the schemewise projective model structure. 
The lemma follows immediately from the observation that ${\rm Ev}_{X}$ preserves
fibrations and trivial fibrations.
\end{proof}

\begin{corollary}\label{cor:DQA-Psh}
Let $\sF$ be a pointed motivic $G$-space. Suppose that $\sF$ is equivariant Nisnevich excisive and is $\A^1$-invariant. 
Then for any pointed simplicial set $K$ and any $X$ in $\Sm^G_S$, 
there is a natural isomorphism
\[
[K, \sF(X)] \iso 
[K \wedge X_+, \sF ]_{\Ho^G_{\bullet}(S)}.
\]
In particular there is an isomorphism $\pi_{n}\sF(X) \iso [S^{n} \wedge X_+, \sF ]_{\Ho^{G}_{\bullet}(S)}$ for any $n$.
\end{corollary}
\begin{proof}
Let $\sF \to \mcal{Q}\sF$ be a motivic injective fibrant replacement. 
By \aref{thm:A1-flasq} it is a schemewise weak equivalence. 
Together with the previous proposition, we have natural isomorphisms
\begin{align*}
[K \wedge X_+, \sF ]_{\Ho^{G}_{\bullet}(S)}
 & \iso   
[{\bf L}{\rm Fr}_X(K), \sF]_{\Ho^{G}_{\bullet}(S)} \\
& \iso 
[K, {\bf R}{\rm Ev}_X (\sF) ] 
 \iso  [K, \mcal{Q}\sF(X) ] 
 \iso  [K, \sF(X) ].
\end{align*}
\end{proof}

\subsection{Equivariant vector bundles}
\label{subsubsection:EBM}
 We finish this section with the observation that
 equivariant vector bundle projections are motivic weak equivalences. Recall that a 
 $G$-equivariant vector bundle $p:\sV\to X$  is an equivariant map of $G$-schemes 
 which is a vector bundle when the $G$-action is forgotten.
An elementary $\A^1$-homotopy between maps $f,g : \sX \to \sY$ of motivic $G$-spaces, 
is an equivariant map $H:\sX\times\A^1\to\sY$ such that $H\circ i_0=f$ and $H\circ i_1=g$. Elementary $\A^1$-homotopic maps become equal in the equivariant motivic homotopy category. 

\begin{prop}\label{prop:Hom-inv}
Let $p:\sV \to X$ be a $G$-equivariant vector bundle. Then $p$ is an equivariant motivic weak equivalence.
\end{prop}
\begin{proof} 
Write $i:X\to \sV$ for the zero-section. Then $p\circ i=\id$ and it suffices to show that there is an elementary $\A^1$-homotopy between $i\circ p$ is and the identity.

Let $\sE\to M$ be a vector bundle over a scheme $M$. Write 
$H_{\sE}:\sE\times\A^1\to \sE$ for the standard fiberwise contraction. Explicitly, if $U=\Spec(R)\subseteq M$ is an open affine subscheme over which $\sE$ becomes trivial, then $H_{\sE|_{U}}$ is the morphism 
$R[X_1,\cdots,X_n] \to R[T,X_1,\cdots ,X_n]$ by $X_j\mapsto TX_j$.

For any morphism of vector bundles $f:\sE \to \sF$ we have $H_{\sE}\circ (f\times 1_{\A^1}) = f\circ H_{\sF}$ and for any map of schemes $g:Y\to X$, $g^{*}H_{\sE} = H_{g^*\sE}$. Consequently, $H_{\sV}$ is equivariant for any equivariant vector bundle $\sV\to X$ and thus provides the desired elementary $\A^1$-homotopy between $i\circ p$ and the identity.

\end{proof}

\section{Equivariant Nisnevich excision and  \texorpdfstring{$K$}{K}-theory}\label{section:Nis-Desc-rep}
In this section, $S$ is a regular Noetherian base scheme of finite Krull dimension. We assume that 
$G \to S$ is a flat group scheme which satisfies the resolution property.
\begin{ResP}\label{ResP}
Every coherent $G$-module on $X$ is the equivariant quotient of a $G$-vector bundle.
\end{ResP}
 
The resolution property holds in rather broad generality. See  \cite[Remark 1.9]{ThomasonDuke56}
 for a list of cases when the resolution property is fulfilled and \cite{Thomason:resolution} for a comprehensive discussion. We mention a few cases  when the resolution property holds for a smooth $G$-scheme $X$ over a regular, Noetherian base $S$:
\begin{enumerate}
\item[(i)] $G$ is a finite constant group scheme,
\item[(ii)] $G$ is reductive,
\item[(iii)] $\dim(S)\leq 1$, $G\to S$ affine.
\end{enumerate}
Under the assumption, we show that equivariant algebraic $K$-theory of smooth schemes over $S$ is representable in the equivariant motivic homotopy category. 
As an application we characterize all equivariantly contractible smooth affine curves with group action, 
and moreover all equivariant vector bundles on such curves. We also establish equivariant Nisnevich excision for certain non-smooth schemes when $G$ is finite.

\subsection{Nisnevich excision}
Let $X$ be a $G$-scheme over $S$. Write 
$\mcal{P}^{G}(X)$ for the exact category of $G$-vector bundles. 
The equivariant algebraic $K$-theory groups are the homotopy groups $K_{i}^{G}(X) := \pi_{i}\mcal{K}(\mcal{P}^{G}(X))$ of the associated $K$-theory space, defined by Waldhausen's $S_{\bullet}$-construction. 
The assignment $X\mapsto \mcal{K}(\mcal{P}^{G}(X))$ is not a presheaf on $\Sch_{S}^{G}$ but only a pseudo-functor. 
We may obtain a presheaf by  a rectification procedure 
to the pseudo-functor $X\mapsto \sP^{G}(X)$. 
Using \cite[Lemma~3.2.6]{Thomason:hocolim} and the rectification procedure explained in \cite[Chapter~5, p.~179]{Jardine:genet}, yields a presheaf of 
simplicial sets $\sK^G$
on $\Sch^G_S$ such that $\pi_{i}\sK^{G}(X) = K_{i}^{G}(X)$ for all $X$.

\begin{theorem}\label{thm:NIS-DESC1}
Let $G\to S$ be a flat algebraic group scheme over a regular Noetherian base scheme $S$  
which satisfies the resolution property for all $X$ in $\Sm_{S}^{G}$. Then $\sK^G$ is equivariant Nisnevich excisive on $\Sm^G_S$. 
\end{theorem}
\begin{proof}
We need to show that if 
\begin{equation*}\label{eqnA}
\xymatrix@C.9pc{ 
W \ar[r] \ar[d] & Y \ar[d]^{f} \\
U \ar@{^{(}->}[r]^{j} & X} 
\end{equation*}
is an equivariant distinguished square in $\Sm^G_S$
then the diagram of simplicial sets
\begin{equation*}\label{eqnB}
\xymatrix@C.9pc{
\sK^G(X) \ar[r]^{j^*} \ar[d]_{f^*} & \sK^G(U) \ar[d] \\
\sK^G(Y) \ar[r] & \sK^G(W)}
\end{equation*}
is homotopy cartesian.

Consider the commutative diagram of fibration sequences (see \cite[Theorem 2.7]{Thomason:GK}):
\begin{equation*}\label{eqn:K-descent-1}
\xymatrix@C1pc{
G^G(X \setminus U) \ar[r] \ar[d] & G^G(X) \ar[r]^{j^*} \ar[d]_{f^*} & G^G(U) \ar[d] \\
G^G(Y \setminus W) \ar[r] & G^G(Y) \ar[r] & G^G(W),}
\end{equation*}
where $G^G(X)$ denotes the $K$-theory of the exact category of equivariant coherent sheaves on a $G$-scheme $X$. By \cite[Theorem 1.8]{ThomasonDuke56}, the resolution property implies that equivariant $G$-theory agrees with equivariant $K$-theory and so it suffices to see that the right-hand square is homotopy cartesian.
But this follows immediately from the fact that $X\setminus U\iso Y\setminus W$.
\end{proof}

\begin{corollary}\label{cor:M-Uns-Rep1}
With hypothesis as in the previous theorem, there is a natural isomorphism
\[
K^G_i(X) 
\iso 
[S^i \wedge X_+, \sK^{G}]_{\Ho^{G}_{\bullet}(S)}
\]
for any $X$ in $\Sm_{S}^{G}$.
\end{corollary}
\begin{proof}
Recall that equivariant $G$-theory agrees with equivariant $K$-theory \cite[Theorem 1.8]{ThomasonDuke56}.
Thus the claim follows immediately from \aref{cor:DQA-Psh} and homotopy invariance for equivariant $G$-theory \cite[Corollary 4.2]{Thomason:GK}.
\end{proof} 
Write $K_{q}^{G}(-)_{GNis}$ for the sheafification in the equivariant Nisnevich topology of the presheaf $X\mapsto K_{q}^{G}(X)$.
\begin{cor}\label{cor:dss}
With hypothesis as in the previous theorem, there is a strongly convergent descent spectral sequence 
\begin{equation*}
E_{2}^{p,q}
=
H^{p}_{GNis}(X,K_{q}^{G}(-)_{GNis})
\Longrightarrow
K_{q-p}^{G}(X)
\end{equation*}
for any $X$ in $\Sm_{S}^{G}$.
\end{cor}
\begin{proof}
The construction of this spectral sequence is exactly as in 
\cite[Section 6.1]{Jardine:genet}. By \aref{cor:M-Uns-Rep1}, the target of the spectral sequence is as displayed. Convergence of the spectral sequence  follows from \aref{cor:Nis-Dim}.
\end{proof}

We note that in some cases we can deduce that equivariant algebraic $K$-theory is equivariant Nisnevich excisive on non-smooth schemes. 
\begin{theorem}
Suppose that $S=k$ is a field. Let $G$ be a finite group of order coprime to ${\rm char}(k)$. 
Then  $\sK^G$ is equivariant Nisnevich excisive on the category of quasi-projective $G$-schemes.
\end{theorem}
\begin{proof}
We need to see that $\sK^{G}$ converts any equivariant distinguished square (\ref{eqn:cd-square}),  in the category of quasi-projective $G$-schemes,
into a homotopy cartesian square.
Our assumption implies that $[Y/G] \to [X/G]$ is a representable morphism of 
tame Deligne-Mumford stacks which admit coarse moduli schemes. 
Hence the theorem follows from \cite[Corollary~3.8]{AO}.
\end{proof}

\subsection{Equivariantly contractible smooth affine curves}\label{subsection:ECAC}
We shall say that a motivic $G$-space $\sX$ is {\sl equivariantly $\A^1$-contractible} if the map $\sX \to pt$ is a motivic weak equivalence. 
A $G$-equivariant vector bundle $\sV$ on $X\in\Sm^G_k$ is called trivial if there is a $G$-representation $V$ such that $\sV=V \times_{k}\ X$.

As an application of the representability of equivariant algebraic $K$-theory, we prove the following geometric result on equivariant vector bundles.

\begin{thm}\label{thm:Curves} 
Let $k$ be an infinite field and let $G = \<\sigma\>$ be a finite cyclic group of order prime to the characteristic of $k$ such that $\mu_{|G|}\subset k$.
Let $X$ be a smooth affine curve over $k$ with $G$-action. 
Then $X$ is equivariantly $\A^1$-contractible if and only if it is isomorphic to a $1$-dimensional linear representation of $G$. 
In particular, 
all $G$-equivariant vector bundles on $X$ are trivial if $X$ is equivariantly $\A^1$-contractible. 
\end{thm}
\begin{proof}
The assertion that a finite-dimensional representation is equivariantly $\A^1$-contractible follows from \aref{prop:Hom-inv}.  
Below we prove the more difficult converse statement.

Suppose that $X$ is equivariantly $\A^1$-contractible.
Since the $G$-action on a smooth scheme is linearizable, 
we can assume there is smooth projective curve $\ov{X} \in \Sm^G_k$ and an open embedding $j:X \inj \ov{X}$ in $\Sm^G_k$. 
Let $f: X \to \Spec(k)$ denote the structure map. 

{\bf {Claim~1:}}
The curve $X$ is rational.
\\
{\sl Proof of claim 1:}
Consider the commutative diagram
\begin{equation}\label{eqn:Curve-1}
\xymatrix@C1.3pc{
K^G_i(k) {\underset{R(G)}\otimes} \ \Z \ar[d]_{f^*} \ar[r]^>>>{f^i_k} & 
K_i(k) \ar[d]^{f^*} \\
K^G_i(X) {\underset{R(G)}\otimes} \ \Z \ar[r]_>>>{f^i_X} & K_i(X)} 
\end{equation}
with forgetful horizontal maps from equivariant to ordinary $K$-theory. 
\aref{cor:M-Uns-Rep1} shows the left vertical arrow is an isomorphism for all $i \ge 0$.
The top horizontal arrow is an isomorphism for all $i \ge 0$ by \cite[Lemma~5.6]{Thomason:trace}. 
Applying these facts for $i=0$, 
we see that the composite map 
$$
K^G_0(X) {\otimes_{R(G)}} \ \Z \to K_0(X) \to \Z
$$ 
is an isomorphism. 
On the other hand, the first map is surjective over $\Z[{1}/{|G|}]$ by \cite[Theorem~1]{Vistoli:finite}.
It follows that $\Pic(X)$ is a torsion group of exponent $|G|$, 
which can happen if and only if $X$ is rational. 
This proves the claim.

{\bf{Claim~2:}}
$X$ is isomorphic (not necessarily equivariantly) to $\A^1$.
\\
{\sl Proof of claim 2:}
Claim~1 implies that $\ov{X} \simeq \P^1_k$.
Inserting $i = 1$ in~\eqref{eqn:Curve-1} shows the composite map 
$$
K^G_1(X) {\otimes_{R(G)}} \ \Z \xrightarrow{f^1_X} K_1(X)\surj \sO^{\times}(X)
$$ 
is just the inclusion $k^{\times} \inj \sO^{\times}(X)$. 
On the other hand,
$f^1_X$ is surjective over $\Z[{1}/{|G|}]$ by \cite[Theorem~1]{Vistoli:finite}. 
It follows that $k^{\times}[{1}/{|G|}] \simeq\sO^{\times}(X)[{1}/{|G|}]$, 
which can happen if and only if $X \simeq \A^1$ as an open subscheme of $\P^1_k$. 

By the above claims, $X$ is the affine line with $G= \<\sigma\>$-action $\sigma(x)=ax+b$ for some fixed $a,b\in k$ with $a^{|G|} = 1$. 
If $b \neq 0$, then $\sigma$ acts on $\A^1$ without fixed points.
This means that the identity map of $\A^1$ gives an element of $[\A^1, \A^1]_{G, \A^1}$ which can not be equivariantly contracted 
to any fixed point.
In particular, 
$\pi^{G, \A^1}_0(X)$ is not constant and hence $X \to \Spec(k)$ is not a motivic weak equivalence, 
which contradicts our assumption. 
We conclude that $b=0$ and $G$ acts linearly on $\A^1$. 

Finally, the claim about the triviality of all $G$-equivariant vector bundles on $X$ follows from the above combined with \cite{Cast} 
and \cite[Theorem~1]{MMP}.
\end{proof}

\begin{exm}\label{exm:Non-contractible}
\aref{thm:Curves} shows that equivariant $\A^1$-contractibility is a strictly stronger condition than ordinary $\A^1$-contractibility, 
as one would expect.
As an example, 
let the cyclic group of order two $G=\<\sigma\>$ act on $\A^1$ by $\sigma(x)=1-x$.
This action is fixed point free and hence not isomorphic to a $G$-representation. 
Thus $\A^1$ equipped with this action is not equivariantly $\A^1$-contractible.
\end{exm}

\begin{remk}\label{remk:High-dim}
One can ask whether the assertion of \aref{thm:Curves} is true in higher dimension as well. 
This seems to be a very difficult question. 
We do not know the answer even when $G$ is trivial and $X$ is a surface.
That is, it is unknown whether an $\A^1$-contractible smooth affine surface is isomorphic to the affine plane. 
It is known,
however, 
that such surfaces do not admit any non-trivial vector bundles.
\end{remk}

\section{Fixed point Nisnevich descent and rational \texorpdfstring{$K$}{K}-theory}
\label{subsection:H-Nis}
Throughout this section, $G$ is a finite constant group scheme over a field $k$.  Equivariant motivic weak equivalences do not always behave as one might expect from ordinary equivariant homotopy theory. For example, equivariant motivic weak equivalences are not detected by fixed points. To remedy this, 
Herrmann \cite{Herrmann} introduces a variant of the equivariant Nisnevich topology. Unfortunately, as he shows in \cite[Proposition 5.3]{Herrmann}, equivariant algebraic $K$-theory does not satisfy descent in this topology. 
Nonetheless, we show in \aref{thm:ieNrationalizedKtheory} below that equivariant algebraic $K$-theory with rational coefficients does satisfy descent in Herrmann's topology.

\subsection{Comparing sheaf cohomologies}
\label{subsubsection:H-Nis*}

We recall several other Nisnevich-type and \'etale topologies on smooth $G$-schemes  and compare the resulting cohomology groups with coefficients in sheaves of $\Q$-vector spaces.

The following topology was introduced by 
P. Herrmann in \cite{Herrmann}, where it was called the ``$H$-Nisnevich topology''. 
It has also been studied by Ben Williams \cite{Williams}. 
\begin{definition}
A collection $\{Y_{i}\to X\}$ of maps in $\Sm_{k}^{G}$ is a \textsl{fixed point Nisnevich cover} if $\{(Y_{i})^{H}\to X^{H}\}$ is a Nisnevich cover in $\Sm_{k}$ for all subgroups $H\subseteq G$.
\end{definition}

\begin{example}
 The map of \aref{exm:Ison-H-1} is a fixed point Nisnevich cover but, as noted there, is not an equivariant Nisnevich cover.
 \end{example}

Recall (see \aref{section:finite}) that if $X$ is a $G$-scheme and $x\in X$, we write $S_{x}$ for the set-theoretic stabilizer. There is an induced homomorphism
$S_{x}\to {\rm Aut}_{k}(k(x))$ and the scheme-theoretic stabilizer $G_{x}$ is the kernel of this map. 
Replacing set-theoretic stabilizers with scheme-theoretic stabilizers in \aref{prop:nischar} leads to the fixed point Nisnevich topology.
\begin{lemma}[{\cite[Lemma 2.12]{Herrmann}}]
 An equivariant \'etale map $f:Y\to X$ is a fixed point Nisnevich cover if and only if for every $x\in X$ there is a $y\in Y$ such that $f$ induces isomorphisms $k(x)\iso k(y)$ and $G_{y}\iso G_{x}$.
\end{lemma}

The fixed point Nisnevich covers define a Grothendieck topology on $\Sm^{G}_{k}$.  We write $\Sm^{G}_{k/{\rm fpNis}}$ for the resulting site.

Recall that for a $G$-scheme $X$, the isotropy group scheme is a group scheme
$G_X$ over $X$ defined by the cartesian square 
\begin{equation*}\label{eqn:Isotropy}
\xymatrix{
G_X \ar@{^{(}->}[r] \ar[d] & G \times X \ar[d]^-{(\mu_X, {\rm id}_X)} \\
X \ar@{^{(}->}[r]^-{\Delta_X} & X \times X.}
\end{equation*}

An equivariant map $f:Y\to X$ is said to be \textsl{isovariant} if it induces an isomorphism $G_{Y}\iso G_{X}\times_{X} Y$. 
A collection $\{f_{i}:X_i \to X\}_{i \in I}$ of equivariant maps is called an isovariant \'etale cover if it is an equivariant \'etale cover such that each $f_{i}$ is isovariant. It is called an isovariant Nisnevich cover if it is
an isovariant {\'e}tale cover which is also a Nisnevich cover.
The isovariant {\'e}tale site
on smooth schemes was introduced by Thomason \cite{ThomasonDuke56}
in order
to prove {\'e}tale descent for Bott-inverted equivariant $K$-theory
with finite coefficients. Its Nisnevich analogue was introduced by Serpe
\cite{Serpe} in an attempt to prove descent theorems for equivariant
algebraic $K$-theory with integral coefficients. (However, some of the results 
of loc.~cit.~need amendments.)
Write 
$\Sm^{G}_{k/{\rm isoNis}}$ and 
$\Sm^{G}_{k/{\rm isoEt}}$
for the resulting 
sites. 

To simplify the comparison of sites, we introduce the following topology.
A \textsl{fixed point \'etale cover} is an equivariant \'etale cover  $\{Y_{i}\to X\}$ such that for any $x\in X$ there is an index $i=i(x)$ and $y\in Y_{i}$ such that $G_{y}\iso G_{x}$. We write $\Sm_{k/fpEt}^{G}$ for the resulting site.

Recall that a {\sl continuous map of sites} $f:\sD \to \sC$ is a
functor $f^{-1} : \sC \to \sD$ between Grothendieck sites such that
for every sheaf $F$ on $\sD$,
the presheaf $f_*(F) = F \circ f^{-1}$ is a sheaf on $\sC$.
If in addition $f^*$ commutes with finite limits, then 
it is called a {\sl morphism of sites}. If $f^{-1}$ commutes with fiber products then $f$ is continuous if and only if it preserves covers. If in addition the topology on $\sD$ is sub-canonical, 
then a continuous map of sites is a morphism of sites 
(see e.g.,~\cite[Remarks~1.1.44, 1.1.45]{MV}).

Every isovariant Nisnevich cover is by definition an isovariant \'etale cover. It is also obviously a fixed point Nisnevich cover. 
By \cite[Corollary 2.13]{Herrmann}, every equivariant Nisnevich cover is also a fixed point Nisnevich cover. The identity functor thus yields a commutative diagram of morphisms of sites
$$
\xymatrix{
\Sm^{G}_{k/fpEt} \ar[r]\ar[d]& \Sm^{G}_{k/fpNis}\ar[r]\ar[d] & \Sm^{G}_{k/Nis} \\
\Sm^{G}_{k/isoEt} \ar[r] & \Sm^{G}_{k/isoNis}.
}
$$
In fact, as we now show, the vertical arrows are equivalences of sites and so we do not need to worry too much about the distinction between the fixed point and isovariant topologies.  
The following property plays in important role in the study of quotients by algebraic group actions.
\begin{defn}
 An equivariant map $f:X\to Y$ is said to be \textsl{stabilizer preserving at $x\in X$} if $f$ induces an isomorphism $G_{x}\cong G_{f(x)}$. If this condition holds for all $x\in X$ then $f$ is said to be \textsl{stabilizer preserving}.
\end{defn}
 
 Note that $f$ is stabilizer preserving if and only if it is isovariant. 
 The notion of a stabilizer preserving map was first introduced by Deligne in unpublished work (see \cite[p.~183]{Knutson}) to prove the existence of a quotient of a separated algebraic space by a finite group and to remedy the problem that an equivariant \'etale map need not induce an \'etale map on the quotients.

\begin{prop}[Rydh]\label{prop:Rydh}
Let $f:X\to Y$ be an equivariant \'etale map. The subset $X_{0}\subseteq X$ of points at which $f$ is stabilizer preserving is an invariant open subset. 
\end{prop}
\begin{proof}
This is a special case of \cite[Proposition 3.5]{Rydh}. 
The main point of the argument is that there are cartesian squares
$$
\xymatrix{
G_{X} \ar@{^{(}->}[r]\ar[d] & G_{Y}\times_{Y} X \ar[d]\ar[r] & G\times Y \ar[dd] \\
X \ar@{^{(}->}[r]^-{\Delta} & X\times_{Y} X\ar[d] & \\
& X \ar[r] & Y\times Y.
}
$$
The locus $X_{0}$ of stabilizer preserving points is the complement of the image of $Z:=G_{Y}\times_{Y} X \setminus G_{X}$.
Since $f$ is \'etale, $\Delta$ is an open immersion and so  $Z\subseteq G_{Y}\times_{Y}X$ is closed. Since $G$ is finite, 
$G\times Y\to Y\times Y$ is  proper and
so $G_{Y}\times_{Y}X\to X$ is as well.
Therefore the image of $Z$ in $X$ is closed and so $X_{0}$ is open.
\end{proof}

\begin{cor}\label{cor:ieiso}
The  identity functor induces equivalences of categories 
 \begin{align*}
& \Shv(\Sm^{G}_{k/fpEt})  \to \Shv(\Sm^{G}_{k/isoEt})\\
& \Shv(\Sm^{G}_{k/fpNis})  \to \Shv(\Sm^{G}_{k/isoNis}).
 \end{align*}
In particular $H^{*}_{fpEt}(X,\mathcal{F}) = H^{*}_{isoEt}(X, \mathcal{F})$ and $H^{*}_{fpNis}(X, \mathcal{F}) = H^{*}_{isoNis}(X, \mathcal{F})$ for any $X$ and any $\mathcal{F}$.
 \end{cor} 
\begin{proof}
 If $X\to Y$ is a fixed point \'etale cover (resp.~a fixed point Nisnevich cover) let $X_{0}\subseteq X$ be the subset of points at which $f$ is stabilizer preserving, which is an open invariant subset by the previous proposition. By the definition of the fixed point \'etale and Nisnevich topologies, $X_{0}\subseteq X\to Y$ is still surjective and so is a cover in the isovariant \'etale (resp.~ Nisnevich) topology. This implies that every fixed point cover can be refined by an isovariant cover,
which establishes the corollary. 
\end{proof}

Recall the description of the points in the equivariant Nisnevich topology \aref{section:finite} for finite groups. There is a corresponding description of the points of the isovariant \'etale topology, which we now detail.
Let $X$ be a $G$-scheme and $x\in X$ a point. Let $\overline{x}=\Spec(\overline{k(x)})\to x$ be geometric point corresponding to a choice of separable closure. We obtain an equivariant map $G/G_{x}\times\overline{x}\to G\cd x$. For notational convenience we define 
$$
G\cd \overline{x} := G/G_{x}\times \overline{x}.
$$
A fixed point \'etale neighborhood of $G\cd\overline{x}\to X$ is an equivariant \'etale map $V\to X$ together with a map 
$G\cd \overline{x}\to  V$ such that the triangle commutes
$$
\xymatrix{
  &  V \ar[d] \\
G\cd \overline{x}\ar[r]\ar[ur] & X.
}
$$
Write $N'_{G}(G\cd\overline{x})$ for the category of affine fixed point \'etale neighborhoods of $G\cd\overline{x}\to X$.
The strict henselization at a geometric point $\overline{x}\to X$ is
 the limit over affine \'etale neighborhoods $V\to X$ of $\overline{x}\to X$. It is functorial on the category of pairs $(Y,\overline{x})$ consisting of a scheme $Y$ and a geometric point $\overline{x}\to Y$ and morphisms of pairs are maps of schemes which preserve the chosen $\overline{x}$-point. Note that $G_{x}$ acts on the pair 
$(X,\overline{x})$ and thus by functoriality, $G_{x}$ acts on 
$\mathcal{O}_{X,\overline{x}}^{h}$.

\begin{prop}
Let $X$ be a $G$-scheme, $x\in X$ a point  and $\overline{x}\to x$ a geometric point corresponding to a separable closure $k(x)\subseteq \overline{k(x)}$. Then there is a natural isomorphism
$$
\lim_{V\in N_{G}^{'}(G\cd\overline{x})} V \cong 
G\times^{G_{x}}\Spec(\mathcal{O}^{h}_{X,\overline{x}}).
$$
\end{prop}
\begin{proof}
One may check that the inclusion $N_{G}^{'}(G\cd\overline{x})\subseteq N(G\cd\overline{x})$ (the category of nonequivariant affine \'etale neighborhoods) is initial.  Therefore we have natural isomorphisms
$$
\lim_{V\in N_{G}^{'}(G\cd\overline{x})}  V \xrightarrow{\cong} 
\lim_{V\in N^{Et}(G\cd\overline{x})} V  \xleftarrow{\cong} 
\lim_{V\in N_{G}'(G\cd\overline{x}\to G\times^{G_{x}}X)} V 
\xleftarrow{\cong} G\times^{G_{x}}\Spec(\mathcal{O}^{h}_{X,\overline{x}}).
$$
\end{proof}

\begin{remark}
 Note that $G\times^{G_x}\Spec(\mathcal{O}^{h}_{X,\overline{x}})$  equals  the limit over isovariant \'etale neighborhoods of $G\cd\overline{x}$ as well.
\end{remark}

As usual, if $F$ is a presheaf and $W = \lim_{i}W_{i}$, then we set 
$F(W) := \colim_{i} F(W_{i})$.  
For each $X$ in $\Sm^{G}_{k}$ and $x\in X$, choose a separable closure $k(x)\subseteq \overline{k(x)}$. This gives rise to the point
$\overline{x}^{*}:\Shv_{isoEt}(\Sm^{G}_{k})\to \Sets$
of the isovariant \'etale topos, defined by 
$\overline{x}^{*}F = F(
G\times^{G_{x}}\Spec(\mathcal{O}_{X,\overline{x}}^{h}))$. 
\begin{prop}
 The set of points $\{\overline{x}^{*}\,|\, x\in X,\, X\in \Sm^{G}_{k}\}$   forms a conservative set of points for $(\Sm^{G}_{k})_{isoEt}$.
\end{prop}
\begin{proof}
 Straightforward and similar to the argument in \aref{prop:Point-Cons} for the Nisnevich topology.
\end{proof}

Since $G$ is finite, a geometric quotient $X/G$ always exists as a separated algebraic space over $k$ 
(see e.g., \cite[Corollary 5.4]{Rydh}). 
Moreover, one sees from \cite[Theorem 2.14]{Kollar} or \cite[Corollary 5.4]{Rydh}
that if $V\to X$ is an \'etale, stabilizer preserving morphism then $V/G\to X/G$ is \'etale and the following square is cartesian
$$
\xymatrix{
V\ar[r]\ar[d] & X \ar[d] \\  
V/G\ar[r] & X/G.  
}
$$
If $U\to X/G$ is a separated, \'etale morphism then $U\times_{X/G} X\to X$ is an \'etale, separated morphism from an algebraic space to a scheme and so $U\times_{X/G}X$ is also a scheme. In particular $\pi:X\to X/G$ induces a functor $\pi^{-1}:(X/G)_{Et}\to X_{isoEt}$ given by $U\mapsto U\times_{X/G}X$. 
Recall (see \aref{section:finite}) that $X^{h}_{Gx}$ is the limit over equivariant Nisnevich neighborhoods of the orbit $G\cd x$. 

\begin{prop}\label{prop:hiso}
 Suppose that the geometric quotient $X\to X/G$ exists as a scheme. Then there is a natural isomorphism $(X^{h}_{Gx})/G \iso (X/G)^{h}_{x}$.
\end{prop}
\begin{proof}
 Write $[x]\in X/G$ for the image of $x$ under the quotient map. Any equivariant Nisnevich neighborhood $f:V\to X$ of $x\in X$ is 	stabilizer preserving at $x$ and so by \aref{prop:Rydh} we may assume it is stabilizer preserving. Therefore $f$ induces an \'etale morphism $f/G:V/G\to X/G$. The section $G\cd x\to V$ induces a section $[x]=(G\cd x)/G\to V/G$ and so $f/G$ is a Nisnevich neighborhood of $[x]$. On the other hand, if $W\to X/G$ is a Nisnevich neighborhood of $[x]$ then $W\times_{X/G}X\to X$ is an equivariant Nisnevich neighborhood of $G\cd x$. It is straightforward to check that these processes are inverse to each other and yield the isomorphism of the proposition.
\end{proof}

\begin{prop}[Thomason]\label{Quotient}
 Let $X$ be a $G$-scheme over $S$ and $\pi:X\to X/G$ be the geometric quotient in algebraic spaces over $k$. 
 Then $ \pi^{-1}:(X/G)_{Et}\to X_{isoEt}$ 
 is
 an equivalence of sites. 
\end{prop}
\begin{proof}
 The proof is similar to that of 
 \cite[Proposition 2.17]{ThomasonDuke56}.  Define a functor $X_{iso-et}\to (X/G)_{et}$ 
 by $V\mapsto V/G$. The discussion above shows that this is  well-defined and is an inverse to $\pi^{-1}$.
 \end{proof}

\begin{cor}\label{cor:Hettorsion}
 Let $X$ be an affine $G$-scheme over $k$, $x\in X$,   and $\mathcal{F}$ a sheaf of abelian groups in the isovariant \'etale topology on $X$.
 Then $H^{p}_{isoEt}(X^{h}_{Gx}, \mathcal{F})$ is torsion for any $p>0$.
\end{cor} 
\begin{proof}
\aref{Quotient} and \aref{prop:hiso} together  imply that we have an isomorphism
$H^{p}_{isoEt}(X^{h}_{Gx}, \mathcal{F}) \cong H^{p}_{Et}((X/G)^{h}_{x}, \pi_*\mathcal{F})$.
Since $(X/G)^{h}_{x}$ is an affine Henselian local scheme, 
$H^{p}_{Et}((X/G)^{h}_{x}, \pi_*\mathcal{F})$ is a torsion group.
\end{proof}

\begin{thm}
\label{thm:sheafcohomologycomparison}
Let $\mathcal{F}$ be a sheaf of $\mathbb{Q}$-modules in the isovariant \'etale topology on $X$. 
The change of topology functors induce natural isomorphisms 
\begin{equation*}
H^{\ast}_{GNis}(X,\mathcal{F})\xrightarrow{\cong}
H^{\ast}_{fpNis}(X,\mathcal{F})
\overset{\cong}{\rightarrow}
H^{\ast}_{isoEt}(X,\mathcal{F}).
\end{equation*}
\end{thm}
\begin{proof}
The change of topology spectral sequence comparing cohomology in the isovariant \'etale topology and in the equivariant Nisnevich topology collapses as a result of \aref{cor:Hettorsion}.
Similarly, the one comparing cohomology in the isovariant \'etale topology and in the fixed point Nisnevich topology, also collapses.  
\end{proof}

\subsection{Presheaves with equivariant \texorpdfstring{$K_{0}$}{K}-transfers}
We introduce an equivariant generalization of the notion of a presheaf with $K_{0}$-transfers. This is a generalization of Voevodsky's notion of a presheaf with transfer, introduced by M. Walker \cite{Walker} (see as well \cite{Suslin:Grayson}), which is particularly well suited for studying $K$-theory.
Write $\mathcal{P}_{G}(X,Y)$ for the category of coherent $G$-modules on $X\times Y$ which are flat over $X$ and whose support is finite over $X$ (as usual we refer briefly to these conditions as \textsl{finite and flat over $X$}).
This category is closed under extensions inside of the abelian category of coherent $G$-modules on $X\times Y$ and so forms an exact category.
Define
$$
K_{0}^{G}(X,Y): = K_{0}(\mathcal{P}_{G}(X,Y)).
$$
Define $K_{0}(\Sm_{k}^{G})$ to be the category whose objects are the same as $\Sm_{S}^{G}$ and $\Hom_{K_{0}(\Sm_{k}^{G})}(X,Y) = K_{0}^{G}(X,Y)$.
For $X_1,X_2,X_3$ in $\Sm_{k}^{G}$, the composition pairing 
$\circ: K_{0}^{G}(X_2,X_3)\times K_{0}^{G}(X_1,X_2) \to K_{0}^{G}(X_1,X_3)$ is induced by 
the pairing of exact categories
$$
\mathcal{P}_{G}(X_{2},X_{3})\times \mathcal{P}_{G}(X_1,X_2) \to \mathcal{P}_{G}(X_1,X_3)
$$
given by $(\mathcal{Q},\mathcal{P})\mapsto(p_{13})_*(p_{12}^*(\mathcal{P})
\otimes p_{23}^{*}(\mathcal{Q}))$. The tensor product is over ${\mathcal{O}_{X_1\times X_2\times X_3}}$
and $p_{ij}:X_1\times X_2\times X_3\to X_{i}\times X_{j}$ is the projection.

\begin{remark}
It is useful to note the following two special cases of composition in 
$K_{0}(\Sm_{k}^{G})$. 
\begin{enumerate}
\item $P\circ f = (f\times\id_{Z})^*(P)$ for a morphism $f:X\to Y$ in $\Sm_{k}^{G}$ and $P\in K_{0}(Y,Z)$.
 \item $ g\circ Q = (\id_{X}\times g)_{*}(Q)$ for a morphism $g:Y\to Z$ in $\Sm_{k}^{G}$ and $Q\in K_{0}(X,Y)$. 
\end{enumerate}
\end{remark}

\begin{defn}
 An \textsl{equivariant $K_{0}$-presheaf} $\mathcal{F}$ on $\Sm_{k}^{G}$ is an additive presheaf $\mathcal{F}:K_{0}(\Sm_{k}^{G})^{op}\to \Ab$.
\end{defn}

There is a functor $\Sm_{k}^{G}\to K_{0}(\Sm_{k}^{G})$ which is the identity on objects and sends a morphism $f:X\to Y$ to the structure sheaf $\mathcal{O}_{\Gamma_{f}}$ of the graph $\Gamma_{f}\subseteq X\times Y$ of $f$. In particular, an equivariant $K_{0}$-presheaf is also a presheaf on $\Sm_{k}^{G}$.

Let $f:X\to Y$ be a finite, flat equivariant map and write $\Gamma_{f}^{t}\subseteq Y\times_{S} X$ for the transpose of the graph of $f$. Its structure sheaf is a morphism $f^{t}:Y\to X$ in $K_{0}(\Sm_{k}^{G})$. For an equivariant $K_0$-presheaf 
$\mathcal{F}$, write $f_* := \mathcal{F}(f^t):\mathcal{F}(X)\to \mathcal{F}(Y)$ for the induced morphism. More generally for any $f:X\to Y$, write $K_{0}^{G}(X/Y)$ for the $K$-theory of the category $\mathcal{P}_{G}(X/Y)$ of coherent $G$-modules on $X$ which are finite and flat over $Y$. 
Any element of $K_{0}^{G}(X/Y)$ give rise to a transfer map in an equivariant $K_0$-presheaf. 
Indeed, there is an exact functor $\mathcal{P}_{G}(X/Y)\to \mathcal{P}_{G}(Y,X)$ given by 
$P\mapsto \langle f, \id_{X} \rangle_*(P)$.

\begin{example}
\begin{enumerate}
\item The category $\mathcal{P}_{G}(X,\Spec(k))$ is the category of $G$-vector bundles on $X$ and so equivariant algebraic $K$-theory $K_{n}^{G}(-)$ is an equivariant $K_{0}$-presheaf  for all $n$. 

\item 
If $\mathcal{F}$ is an equivariant $K_{0}$-presheaf then $\mathcal{F}(X)$ is a module over $K_{0}^{G}(X)$ as follows.
There is an exact functor $\mathcal{P}_{G}(X) \to \mathcal{P}_{G}(X,X)$ given by $\mathcal{V}\mapsto \Delta_{*}\mathcal{V}$ which induces the map 
$K_{0}^{G}(X)\to K_{0}^{G}(X,X)$. It is straightforward to check that defining $[\mathcal{V}]\cdot x = \mathcal{F}([\Delta_*\mathcal{V}])(x)$ for 
$x\in \mathcal{F}(X)$, $[\mathcal{V}]\in K_{0}(X)$ equips $\mathcal{F}(X)$ with the desired module structure.
\end{enumerate}
\end{example}

\begin{prop}\label{prop:transfer}
 Let $\mathcal{F}$ be an equivariant $K_{0}$-presheaf. 
 \begin{enumerate}
  \item Let $f:X\to Y$ be a finite, flat equivariant morphism. Then we have that  $f_*f^*:\mathcal{F}(Y)\to \mathcal{F}(X)\to \mathcal{F}(Y)$ is multiplication by $[f_*\mathcal{O}_{X}]\in K_{0}^{G}(Y)$. 
  \item Let $f:X\to Y$ be an isovariant finite \'etale map of degree $d$. Suppose that a geometric quotient $Y\to Y/G$ exists in $\Sch_{k}$. Then $f^*f_*:\mathcal{F}(Y)\to \mathcal{F}(Y)$ is equal to multiplication by the degree of $f$. 
 \end{enumerate}
\end{prop}
\begin{proof}
 Unraveling the definitions, we see that $f_*f^*$ is the map induced by the endomorphism $[\Delta_*f_*\mathcal{O}_{X}]\in K_{0}^{G}(Y,Y)$ which establishes the first item. For the second item we note that the hypothesis implies that $X\to X/G$ exists in $\Sch_{k}$ and that we have a cartesian square (see the discussion preceding \aref{prop:hiso})
 $$
 \xymatrix{
 X \ar[r]^{f}\ar[d]_{\pi_{X}} & Y \ar[d]^{\pi_{Y}} \\
 X/G \ar[r]^{\overline{f}} & Y/G .
 }
 $$
 The map $\overline{f}$ is finite \'etale of degree equal to the degree of $f$ and in $K_{0}^{G}(Y)$ we have that $[f_*\mathcal{O}_{X}] = [f_*(\pi_{X})^*\mathcal{O}_{X/G}] = [\pi_{Y}^*\overline{f}_*\mathcal{O}_{X/G}]$. It suffices to see that $[\overline{f}_{*}\mathcal{O}_{X/G}]$ is equal to the degree of $\overline{f}$ in $K_{0}^{G}(Y/G)$. This follows from the fact that if $M$ has trivial action then $K_{0}^{G}(M) = K_0(M)\otimes K^G_0(k)$.
\end{proof}

\begin{cor}\label{cor:shfvanish}
 Let $\mathcal{F}$ be an equivariant $K_{0}$-presheaf of $\Q$-modules. If the sheafification $\mathcal{F}_{isoEt} = 0$ then $\mathcal{F}_{GNis}=0$ as well.
\end{cor}
\begin{proof}
Let $X$ be a smooth $G$-scheme over $S$, $x\in X$, and $c\in \mathcal{F}(X^{h}_{Gx})$. Since  
$\mathcal{F}_{isoEt} = 0$ there is a 
finite, isovariant \'etale morphism $f:V\to  X^{h}_{Gx}$ such that $f^{*}(c) = 0$. But $f_{*}f^*(c) = \deg(f)\cdot c$ by  \aref{prop:transfer} and therefore $c =0$.

\end{proof}

If $\tau$ is a Grothendieck topology on $\Sm_{S}^{G}$ then $\mathcal{F}$ is said to be a \textsl{$\tau$-sheaf with equivariant $K_{0}$-transfers}, or an \textsl{equivariant $K_{0}$-$\tau$-sheaf} for short,  if it is an equivariant $K_{0}$-presheaf  whose underlying presheaf on $\Sm_{S}^{G}$ is a $\tau$-sheaf. 

\begin{lem}\label{lem:nistran}
 Let $f:U\to Y$ be an equivariant Nisnevich cover (resp.~ an isovariant \'etale cover) and $P\in K_{0}(X,Y)$. Then there is an equivariant Nisnevich cover (resp.~ an isovariant \'etale cover) $f':V\to X$ and $Q\in K_{0}(V,Y)$ which fit into a commutative square in $K_{0}(\Sm_{S}^{G})$
 $$
 \xymatrix{
 V \ar[r]^{Q}\ar[d]_{f'} & U \ar[d]^{f} \\
 X\ar[r]^{P} & Y.
 }
 $$
\end{lem}
\begin{proof}
We treat the case of an equivariant Nisnevich cover; the isovariant \'etale case is similar. It suffices to treat the case when $P\in \mathcal{P}_{G}(X,Y)$. Write $Z = \Supp(P)$ and consider the pullback $Z'=U\times_{Y}Z$. Then $Z'\to Z$ is an equivariant Nisnevich cover and $\pi:Z\to X$ is finite. We can find an equivariant Nisnevich cover $V\to X$ such that $V\times_{X}Z'\to V\times_{X} Z$ has an equivariant section. 
Indeed, for any $x\in X$, $Z_{x}=X^{h}_{G x}\times_{X} Z$ is disjoint union of semilocal Henselian affine $G$-schemes with a single orbit and $Z'_{x}=X^{h}_{G x}\times_{X} Z'\to Z_{x}$ is an equivariant Nisnevich cover. Therefore $Z_{x}'\to Z_{x}$ has an equivariant section and so there is some equivariant Nisnevich neighborhood $V_{x}\to X$ of $G x$ such that   $V_{x}\times_{X}Z'\to V_{x}\times_{X} Z$ has an equivariant section. The covering $\{V_{x}\to X\}$ has a finite subcovering $\{V_{x_1},\ldots, V_{x_n}\}$ and $V:=\coprod V_{x_i}\to X$ has the property that $V\times_{X}Z'\to V\times_{X} Z$ has an equivariant section.

Now let $s:V\times_{X}Z\to V\times_{X} Z'$ be a choice of equivariant  section and write $j: s(V\times_{X}Z)\hookrightarrow V\times U$ for the resulting inclusion (which is a closed invariant subscheme that is finite over $V$). Now set 
$Q= j_*s_*P|_{V\times_{X}Z}$. Then $\Supp(Q)= s(V\times_{X}Z)$, $Q$ is flat over $V$ and $P\circ f' = f\circ Q$ in $K_{0}(\Sm_{S}^{G})$.
\end{proof}

\begin{thm}
\label{thm:eK0presheaves1}
If $\mathcal{F}$ is an equivariant $K_{0}$-presheaf then the equivariant Nisnevich sheafification (resp.~ isovariant \'etale sheafification) has a unique structure of an equivariant $K_{0}$-presheaf such that $\phi:\mathcal{F}\to (\mathcal{F})_{GNis}$ (resp.~ $\phi:\mathcal{F}\to (\mathcal{F})_{isoEt}$) is a morphism of equivariant $K_{0}$-presheaves.
\end{thm}
\begin{proof}
This is similar to the nonequivariant case 
(see e.g., \cite[Lemma 1.5]{Suslin:Grayson}). We treat the case of the equivariant Nisnevich topology; the case of the  isovariant \'etale topology is similar.

 We begin with uniqueness. Let $\mathcal{F}_{1}$ and $\mathcal{F}_{2}$ be two equivariant $K_{0}$-presheaves  with a map of equivariant $K_{0}$-presheaves 
 $\mathcal{F}\to \mathcal{F}_{i}$ whose underlying map of presheaves is the canonical map 
 $\mathcal{F}\to \mathcal{F}_{GNis}$. 
 Let $P:X\to Y$ be a map in $K_{0}^{G}(\Sm_{S}^{G})$ and 
 $y\in \mathcal{F}_{1}(Y) = \mathcal{F}_{2}(Y) = \mathcal{F}_{GNis}(Y)$. Choose an equivariant Nisnevich covering $U\to Y$ such that $y|_{U}$ is in the image of $u\in F(U)$. Applying  \aref{lem:nistran} we have a commutative square in $K_{0}(\Sm_{S}^{G})$
$$
\xymatrix{
V\ar[d]_{f'}\ar[r]^{Q} & U \ar[d]^{f} \\
X\ar[r]^{P} & Y 
}
$$
where $f':V\to X$ is an equivariant Nisnevich cover. It is straightforward to verify, using this square, that $\mathcal{F}_{1}(P)(y) = \mathcal{F}_{2}(P)(y)$ and so $\mathcal{F}_{1}=\mathcal{F}_{2}$ as equivariant $K_{0}$-presheaves.

Now we show existence. 
First we note that by \aref{lem:nistran}, if 
 $P\in K_{0}^{G}(X,Y)$ and $y\in \mathcal{F}(Y)$ is a section which vanishes in $ (\mathcal{F}(Y))_{GNis}$ then $(P^*y)$ vanishes in $(\mathcal{F}(X))_{GNis}$ as well. This implies that the separated (in the equivariant Nisnevich topology) presheaf $s_{GNis}\mathcal{F}$ has the structure of an equivariant $K_{0}$-presheaf such that $\mathcal{F}\to s_{GNis}\mathcal{F}$ is a morphism of equivariant 
 $K_{0}$-presheaves.
 We may therefore assume that $\mathcal{F}$ is a separated presheaf. 
Let $P:X\to Y$ be a morphism in $K_{0}(\Sm_{S}^{G})$ and $y\in \mathcal{F}_{GNis}(Y)$. We need to define $\mathcal{F}(P)(y)\in \mathcal{F}(X)$. There is an equivariant Nisnevich cover $f:U\to Y$ such that $y|_{U}$ is the image of $u\in \mathcal{F}(U)$. By \aref{lem:nistran} there is an equivariant Nisnevich cover $f':V\to X$ and $Q\in K_{0}^{G}(V,U)$ such that $f\circ Q = P\circ f'$. Consider 
$x'=\mathcal{F}(Q)(u)$. Write $\pi_i:U\times_{Y} U\to U$ and $\pi_{i}':V\times_{X}V\to V$, $i=1,2$ for the projection to the $i$th factor. Note that $\pi_1^*u=\pi_2^*u$ and this implies that $(\pi_{1}')^*(x') = (\pi_2')^*(x')$. Thus $x'$ determines an element $x\in \mathcal{F}(X)$ and define
$\mathcal{F}(P)(y) := x$. It is straightforward to check that this endows $\mathcal{F}$ with the structure of an equivariant $K_{0}$-presheaf  and $\mathcal{F}\to (\mathcal{F})_{GNis}$ is a morphism of equivariant $K_{0}$-presheaves. 
\end{proof}

\begin{cor}
\label{cor:eK0presheaves2}
Let $\mathcal{F}$ be a $K_0$-presheaf of $\Q$-modules. Then $(\mathcal{F})_{GNis}\to (\mathcal{F})_{isoEt}$ is an isomorphism.
\end{cor}
\begin{proof}
Write $\mathcal{G}$ for the presheaf kernel or cokernel of $\mathcal{F}\to (\mathcal{F})_{isoEt}$.  By the previous theorem these are equivariant $K_{0}$-presheaves and so $\mathcal{G}$ is as well.
Since $\mathcal{G}_{isoEt}=0$, \aref{cor:shfvanish} implies that $\mathcal{G}_{GNis} = 0$ as well. 
\end{proof}

\subsection{Descent for rationalized equivariant \texorpdfstring{$K$}{K}-theory}
Now we show that rationalized equivariant algebraic $K$-theory $\mathcal{K}_{G}(-)_{\Q}$ satisfies descent in the isovariant \'etale topology.
Let $\tau$ be a Grothendieck topology on the category $\mathcal{C}$ which we assume has enough points. Let $\mathcal{F}$ be a presheaf of spectra on $\mathcal{C}$. Write $\mathcal{Q}_{\tau}$ for a fibrant replacement functor in any of the $\tau$-local model structures on presheaves of spectra on $\mathcal{C}$. 
Replacing the \'etale topology by the $\tau$-topology in the construction in \cite[Section 6.1]{Jardine:genet}  leads to a
conditionally convergent spectral sequence  
\begin{equation*}
E_{2}^{p,q}
=
H^{p}_{\tau}(X,(\pi_{q}\mathcal{F})_{\tau})
\Longrightarrow
\pi_{q-p}\mathcal{Q}_{\tau}\mathcal{F}
\end{equation*}
called the \textsl{{$\tau$}-descent spectral sequence}. 

\begin{thm}
\label{thm:ieNrationalizedKtheory}
Let $k$ be a field and $G$  a finite group.
The rationalized $G$-equivariant $K$-theory presheaf 
$\mathcal{K}_{G}(-)_{\mathbb{Q}}$ satisfies descent in the fixed point Nisnevich and in the isovariant \'etale topologies on $\Sm^{G}_{k}$. 
\end{thm}
\begin{proof}
We compare the descent spectral sequences for equivariant $K$-theory in the equivariant Nisnevich, fixed point Nisnevich, and isovariant \'etale topologies
$$
\xymatrix{
E_{2}^{p,q}
=
H^{p}_{GNis}(X, (K_{q}^{G}(-)_{\mathbb{Q}})_{GNis})
\Longrightarrow
\pi_{q-p}\mathcal{Q}_{GNis}\mathcal{K}_{G}(X)_{\mathbb{Q}}\ar[d] \\
E_{2}^{p,q}
=
H^{p}_{fpNis}(X, (K_{q}^{G}(-)_{\mathbb{Q}})_{fpNis})
\Longrightarrow
\pi_{q-p}\mathcal{Q}_{fpNis}\mathcal{K}_{G}(X)_{\mathbb{Q}}\ar[d] \\
E_{2}^{p,q}
=
H^{p}_{isoEt}(X, (K_{q}^{G}(-)_{\mathbb{Q}})_{isoEt})
\Longrightarrow 
\pi_{q-p}\mathcal{Q}_{isoEt}\mathcal{K}_{G}(X)_{\mathbb{Q}}.
}
$$
  \aref{thm:sheafcohomologycomparison} and  \aref{cor:eK0presheaves2} combined imply that on the $E_{2}$-pages the vertical arrows are isomorphisms. 
Since these are comparisons of conditionally convergent spectral sequences, we conclude that 

\begin{equation*}
\pi_{q-p}\mathcal{Q}_{GNis}\mathcal{K}_{G}(X)_{\Q} \xrightarrow{\cong} 
\pi_{q-p}\mathcal{Q}_{fpNis}\mathcal{K}_{G}(X)_{\Q} \xrightarrow{\cong}
\pi_{q-p}\mathcal{Q}_{isoEt}\mathcal{K}_{G}(X)_{\Q}.
\end{equation*}
Since 
 equivariant algebraic $K$-theory satisfies descent in the equivariant Nisnevich topology
 we have an isomorphism $K_{q-p}^{G}(X)_{\Q} =
 \pi_{q-p}\mathcal{Q}_{GNis}\mathcal{K}_{G}(X)_{\Q}$ and the result follows.
\end{proof}

\section{Equivariant homotopical purity and blow-up theorems}
\label{section:EPT}
Throughout this section, $k$ is a perfect field and $G$ is a finite constant group scheme whose order is coprime to ${\rm char}(k)$.
The homotopy purity theorem \cite[Theorem~3.2.23]{MV} is one of the 
most important tools in motivic homotopy theory, 
e.g., in the construction of Gysin long exact sequences and for Poincar{\'e} 
duality.
The {\sl equivariant Thom space} of a $G$-equivariant vector bundle $E \to X$ in $\Sm^G_S$ 
is the pointed motivic $G$-space 
$$
\Th(E):= E/{(E \setminus X)}
$$
where $X \inj E$ is the zero section. The equivariant version of the homotopical purity theorem is the assertion that if $Z\subseteq X$ is a closed, invariant smooth subscheme of a smooth $G$-scheme $X$ then there is a natural isomorphism in 
${\Ho}^G_{\bullet}(k)$
$$
X/X\setminus Z \iso \Th(N_{Z/X}).
$$

We show in \aref{thm:Purity} below that the equivariant homotopical purity theorem holds when $G$ is abelian and $k$ has enough roots of unity. The method of proof is an equivariant version of Morel-Voevodsky's argument in \cite{MV} in the nonequivariant case. As such, the key geometric input we need to establish the equivariant homotopical purity theorem is
that locally a closed inclusion of smooth $G$-schemes looks like an inclusion of representations. This is a delicate statement as can be seen by contemplating tangent representations.
Given a point $x\in X$, the tangent space of $X$ at $x$ is the $k(x)$-vector space $T_xX := \Hom_{k(x)}(\Omega_{X/k,x}\otimes k(x), k(x))$. Note that if $X$ has a $G$-action, then there is an induced $k(x)$-linear action of the stabilizer $G_{x}$ on $T_xX$, i.e., the tangent space $T_{x}X$ is naturally a $G_{x}$-representation over $k(x)$ for any $x\in X$. In fact, $T_{x}X$ has even more structure, namely that of a module over the twisted group ring $k(x)^{\#}[S_{x}]$. The subtleties that arise in establishing local linearization of smooth pairs (and hence in establishing the equivariant homotopical purity theorem) arise from these extra structures and the fact that linearizations in the equivariant Nisnevich topology are sensitive to them.

\subsection{Linearization of smooth pairs}

Recall that the \textsl{exponent}  of a finite group $G$ is the least common multiple of the orders of its elements. 

\begin{lem}[{\cite[Lemma 8.10]{HVO}}]\label{lem:HVO}
 Let $k$ be a perfect field and $G$ an abelian group whose order is coprime to $k$ and suppose that $k$ is a perfect field which contains a primitive $d$th-root of unity, where $d$ is the exponent of $G$. Let $Z\hookrightarrow X$ be an equivariant closed embedding of smooth affine 
 $G$-schemes over $k$ and let $x\in Z$ be a closed point.  
 Then there are $G$-representations $W_1$, $W_2$, an embedding of representations $W_{2}\subseteq W_1$,  an invariant open neighborhood $U$ of $x$, and an equivariant cartesian diagram 
 $$
 \xymatrix{
 U\cap Z \ar[r]\ar[d] & U \ar[d]^{f} \\
 W_{2} \ar[r] & W_{1}
 }
 $$
 such that $f$ is \'etale.
\end{lem}
\begin{proof}
 This is \cite[Lemma 8.10]{HVO}. In the beginning of the proof of Theorem 8.11 of loc.~cit.~ it is verified that the hypothesis of 
 the cited Lemma 8.10 are satisfied 
 when all irreducible $k[G]$-modules are one dimensional. By a classical theorem of Brauer 
 \cite[Theorem 41.1, Corollary 70.24]{CR}, the condition that $k$ contains a primitive $d$th root of unity implies all irreducible $k[G]$-modules are one dimensional.
\end{proof}

Let $ Z \inj X$ be an invariant closed subscheme of a $G$-scheme. 
An {\sl equivariant Nisnevich neighborhood} of $(X, Z)$ is a commutative square in $\Sch^G_S$
\begin{equation*}
\xymatrix@C1pc{
Z' \ar[r]^{i'} \ar[d]_{\iso} & U \ar[d]^{f} \\
Z \ar[r]_{i} & X,}
\end{equation*}
where $f$ is an equivariant {\'e}tale map. 
We denote such a neighborhood simply by $(U, Z)$. 
If this square is cartesian we call $(U, Z)$ a {\sl distinguished} equivariant Nisnevich neighborhood of $(X, Z)$.

\begin{defn}\label{defn:EN-Linearization}
Let $Z \inj X$ be a closed immersion in $\Sm^G_k$.
An \textsl{equivariant Nisnevich linearization} of the pair $(X, Z)$ consists of a smooth $G$-scheme $U$, an equivariant closed immersion $Z\subseteq U$, a $G$-equivariant vector bundle $E\to Z$, and
a pair of equivariant \'etale maps $p:U\to X$ and $q:U\to E$ such that
\begin{equation*}
(X, Z) \xleftarrow{p} (U, Z) \xrightarrow{q} (E, Z)
\end{equation*}
are both distinguished equivariant Nisnevich neighborhoods (here $Z\subseteq E$ is the zero-section).
\end{defn}   

\begin{prop}\label{prop:Lin-abpt}
Let $Z \inj X$ be a closed immersion in $\Sm^G_k$  with $X$ quasi-projective, 
and $G$ and $k$ as in the previous lemma. Let $x\in Z$ be a closed point.
There is a $G$-invariant open neighborhood $U\subseteq X$ of $x$ such that the
pair $(U, U\cap Z)$ admits an equivariant Nisnevich linearization.
\end{prop}
\begin{proof}
The construction proceeds as in the nonequivariant case in \cite{MV}.  
Since $X$ is quasi-projective, $x$ has an invariant affine neighborhood and so we may assume that $X$ is affine.  
By \aref{lem:HVO}, we may further shrink $X$ equivariantly around $x$ and assume that we have a $G$-equivariant cartesian square 
 $$
 \xymatrix{
 Z\ar[r]\ar[d]_{f'} & X \ar[d]^{f} \\
 W_2\ar[r] & W_1,
 }
 $$  
 where $W_2\subseteq W_1$ is an inclusion of $G$-representations and the vertical maps are \'etale. Write $N:=W_1/W_2$ for the quotient representation. The $G$-representation $W_1$ is isomorphic to a direct sum $W_1 = N \times W_2$. We thus obtain a $G$-equivariant 
\'etale map $f'\times \id: Z\times N \to W_1$. Define $X':=X\times_{W_1}(Z\times N)$. Then 
 $X'\to X$ and $X'\to Z\times N$ are $G$-equivariant, \'etale maps. The preimages of $Z$ and $Z\times 0$ coincide and are equal to $Z':= Z\times_{W_2}Z$. The equivariant, \'etale map $Z'\to Z$ has an equivariant section, given by the diagonal $\Delta_{Z}$, which implies an equivariant decomposition $Z' = \Delta(Z) \coprod C$. 
 Now define $X'':= X'\setminus C$ which is an invariant open subscheme of $X'$. The induced maps $X'' \to X$ and $X''\to Z\times N$ are equivariant and \'etale. Moreover, the preimage of $Z$ and $Z\times 0$ are both equal to $Z$. Therefore
 $$
 (X, Z) \leftarrow (X'', Z)  \rightarrow (Z\times N, Z)
 $$
yields the desired equivariant Nisnevich linearization. 
\end{proof}

\subsection{Deformation to the normal cone}
Let $B(X,Z)$ denote the blow-up of $X \times \A^1$ along the
$G$-invariant closed subscheme $Z \times \{0\}$ (where as usual $\A^{1}$ is considered to have trivial action). 
It is straightforward to check that the $G$-action on $X$ induces one on the smooth scheme $B(X,Z)$ and that the blow-up map
$f: B(X,Z) \to X \times \A^1$ is equivariant. 
There are inclusions of closed pairs
in $\Sm^G_S$
\begin{equation}\label{eqn:Blow-up}
(X, Z) \xrightarrow{i_1} (B(X, Z), Z \times \A^1) \xleftarrow{i_0}
(\P(N_{Z/X} \times \A^1), Z).
\end{equation}
Here the inclusion in the last pair is 
$Z \inj N_{Z/X} = \P(N_{Z/X} \times \A^1) \setminus \P(N_{Z/X})$.
If $E$ is an equivariant vector bundle on $Y$, 
then by a straightforward equivariant version of \cite[Proposition~3.2.17]{MV}, 
we have
$$
\Th(E) \simeq \frac{\P(E \times \A^1)}{
\P(E \times \A^1) \setminus Y}.
$$ 
Therefore from the morphisms of pairs above, we obtain  monomorphisms of pointed motivic $G$-spaces
\begin{equation}\label{eqn:Blow-up-1}
\alpha_{X,Z} : \frac{X}{X \setminus Z} \to \frac{B(X, Z)}{B(X, Z) \setminus
(Z \times \A^1)} ;
\end{equation}
\[
\beta_{X, Z} : {\rm Th}(N_{Z/X}) \to
\frac{B(X, Z)}{B(X, Z) \setminus (Z \times \A^1)}.
\] 

We now state our equivariant homotopical purity theorem.
We anticipate that the equivariant homotopical purity theorem remains valid as long as the order of the group is coprime to the characteristic of $k$.
However, 
as mentioned in the beginning of this section,
the equivariant linearization techniques we use are delicate and require additional hypotheses.

\begin{thm}\label{thm:Purity}
Let $k$ be a perfect field and $G$ a finite abelian group whose order is prime to 
$\mathrm{char}(k)$. Suppose further that $k$ contains a primitive $d$th root of unity, where $d$ is the exponent of $G$.
Then for any closed immersion $Z\hookrightarrow X$  in $\Sm^{G}_{k}$,  
the maps $\alpha_{X,Z}$ and $\beta_{X,Z}$ are equivariant motivic weak equivalences. In particular, there is a canonical isomorphism in 
${\Ho}^G_{\bullet}(k)$ of pointed motivic $G$-spaces
\[
X/{(X \setminus Z)} 
\xrightarrow{\iso}
{\rm Th}(N_{Z/X}).
\]
\end{thm}
From the construction, the equivariant purity isomorphism has the following naturality property.
\begin{proposition}
Let $Z\subseteq X$ be a closed immersion in $\Sm_{k}^{G}$ and
 $f:X^{\prime}\rightarrow X$ a map in $\Sm_k^{G}$ such that $Z^{\prime}=f^{-1}(Z)$ is smooth and
$\sN_{Z'/X'}\rightarrow f^{\ast}\sN_{Z/X}$ is an isomorphism. 
Then $f$ induces a commutative square in $\Ho_{\bullet}^{G}(k)$, where the vertical arrows are the isomorphisms from the previous theorem
\begin{equation*}
\xymatrix{
X^{\prime}/(X^{\prime}\setminus Z^{\prime}) 
\ar[r]\ar[d]_{\iso} & 
X/(X\setminus Z) \ar[d]^{\iso} \\
\Th(\sN_{Z'/X'}) \ar[r] & 
\Th(\sN_{Z/X}).
}
\end{equation*}
\end{proposition}

The proof of \aref{thm:Purity} will occupy the rest of this section.

\subsection{Purity for vector bundles}

For the moment, we let $S$ be a general finite dimensional Noetherian base scheme and $G$ a reductive group scheme over $S$.

\begin{lem}\label{lem:Thom-VB}
Let $Z\inj V$ be the zero section of a $G$-equivariant vector bundle $ V \to Z$  in $\Sm^G_S$. 
Then the maps $\alpha_{V,Z}$ and $\beta_{V,Z}$ are motivic weak equivalences.
\end{lem}
\begin{proof}
We first recall that there is a natural map 
$\lambda_Z: B(V,Z) \to \P(V \times \A^1)$ 
which identifies $B(V,Z)$ with the total space of the relative line bundle $\sO(1)$.
Moreover, one has $\lambda_Z^{-1}(\P(V \times \A^1) \setminus Z) = 
B(V,Z) \setminus (Z \times \A^1)$,
where $Z   \inj \P(V \times \A^1)$ is induced by the zero section
$Z  \inj V = \P(V \times \A^1) \setminus \P(V)$.
In particular, these maps are motivic weak equivalences by
\aref{prop:Hom-inv}. We conclude that the map
\[
q: \frac{B(V,Z)}{B(V,Z) \setminus (Z \times \A^1)} \to
\frac{\P(V \times \A^1)}{\P(V \times \A^1) \setminus Z}
\]
is a motivic weak equivalence.
On the other hand, the composite $q \circ \alpha_{V,Z}$ is a canonical
equivalence of pointed motivic $G$-spaces (see \cite[Proposition~3.2.17]{MV}). 
We conclude that $\alpha_{V, Z}$ is a motivic weak equivalence.

On the other hand, the composition of the projection $\lambda_Z$ with
the inclusion $V \xrightarrow{i_0} B(V,Z)$ is the canonical open inclusion
$V \inj \P(V \times \A^1)$. Since
\begin{equation}\label{eqn:Thom-VB-1}
\xymatrix@C1pc{
V \setminus Z \ar[r] \ar[d] & V \ar[d] \\
\P(V \times \A^1) \setminus Z \ar[r] & \P(V \times \A^1)}
\end{equation}
is a distinguished equivariant Nisnevich square, it follows that the composition
\[
\frac{V}{V \setminus Z} \xrightarrow{\beta_{V,Z}} 
\frac{B(V,Z)}{B(V,Z) \setminus (Z \times \A^1)} \xrightarrow{q}
\frac{\P(V \times \A^1)}{\P(V \times \A^1) \setminus Z}
\]
is a local weak equivalence. Since $q$ is a motivic weak equivalence,
we conclude that $\beta_{V,Z}$ is a motivic weak equivalence.
\end{proof}

\begin{remk}
Purity for vector bundles holds also in the motivic homotopy theory of Deligne-Mumford stacks by the same argument. 
\end{remk}
\begin{remark}
Note that in the above situation, the equivariant purity isomorphism $V/(V\setminus Z) \iso \Th(\sN_{Z/V})$ coincides with the map defined by
the natural isomorphism $V\iso\sN_{Z/V}$.
\end{remark}
\subsection{Purity in general}
Let $G$ and $k$ be as in  \aref{thm:Purity} and 
$Z\inj X$ be a closed immersion in $\Sm^G_k$. Suppose that $f:U\to X$ is an equivariant Nisnevich cover and set $Z_{U}:=U\times_{X}Z$.
Let $\sU\to X$ and $\sZ\to Z$
denote the associated \v{C}ech resolutions. 
That is, $\sU$  is 
the defined motivic $G$-space defined by    
 $\sU_{n} = U \times_{X} \cdots
\times_{X} U$  
and similarly for $\sZ$. 
This yields a morphism of pairs  
$f: (\sU, \sZ) \to (X,Z)$.
Now let $\sB$ be the motivic $G$-space obtained by setting $\sB_{n} = B(\sU_{n},\sZ_{n})$ and similarly
 ${\rm Th}(N_{{\sZ}/{\sU}})$ denotes the motivic $G$-space which is
 defined to be
the levelwise Thom space: ${\rm Th}(N_{{\sZ}/{\sU}})_{n} = 
{\rm Th}(N_{{\sZ_n}/{\sU_n}})$. 
These motivic $G$-spaces fit into the commutative diagram
\begin{equation*}
\xymatrix@C1pc{
\dfrac{\sU}{\sU \setminus \sZ} \ar[r] \ar[d] &
\dfrac{\sB}{\sB \setminus (\sZ\times \A^1)} \ar[d] &
{\rm Th}(N_{{\sZ}/{\sU}}) \ar[d] \ar[l] \\
\dfrac{X}{X \setminus Z} \ar[r] &
\dfrac{B(X, Z)}{B(X, Z) \setminus (Z \times \A^1)}  &
{\rm Th}(N_{Z/X}). \ar[l]}  \\
\end{equation*}

\begin{lem}\label{lem:Purity-we}
The vertical arrows in the above diagram are equivariant Nisnevich local weak equivalences.
\end{lem}
\begin{proof}
Note that $\sB = B(X,Z)\times_{X}\sU$ and 
$N_{\sZ/\sU} = N_{Z/X}\times_{X}\sU$. Since $f:U\to X$ is an equivariant Nisnevich cover, the \v{C}ech resolution $\sU\to X$ is a local weak equivalence  and similarly $\sB \to B$ and $N_{\sZ/\sU}\to N_{Z/X}$ are local weak equivalences as well. For the same reason, the maps $\sU\setminus \sZ \to X\setminus Z$, $\sB\setminus (\sZ\times\A^1)\to B(X,Z)\setminus Z\times \A^1$, and $N_{\sZ/\sU}\setminus \sZ \to N_{Z/X}\setminus Z$ are all local equivalences as well. That the vertical arrows are local equivalences follows since the local model structure is proper. 
\end{proof}

\begin{cor}\label{cor:EHP-reduction1}
 Let $f:U\to X$ be an equivariant Nisnevich cover.  \aref{thm:Purity} holds for the pair $(X,Z)$ if and only if it holds for the pair $(U, f^{-1}Z)$. 
\end{cor}

\begin{cor}\label{cor:EHP-reduction2}
 Suppose that $(X,Z)$ admits an equivariant Nisnevich linearization. Then $\alpha_{X,Z}$ and $\beta_{X,Z}$ are equivariant motivic weak equivalences.
\end{cor}
\begin{proof}
 There are morphisms $(X,Z) \leftarrow (U,Z) \rightarrow (E,Z)$ which are distinguished equivariant Nisnevich neighborhoods.  The morphisms $\alpha_{E,Z}$ and $\beta_{E,Z}$ are equivariant motivic weak equivalences by  \aref{lem:Thom-VB} and so this follows from the previous corollary.
\end{proof}

{\sl Proof of \aref{thm:Purity}:}
There is an equivariant Nisnevich cover $Y\to X$ such that 
$Y$ is a smooth affine $G$-scheme (see \aref{lem:affinenbd}). 
By  \aref{prop:Lin-abpt}, every closed point of $Y$ has an invariant open neighborhood which admits an equivariant Nisnevich linearization. Let $U_1,\ldots, U_{r}$ be finitely many such invariant open neighborhoods which cover $Y$. Now we set $U :=\coprod U_{i}$ and write $f:U\to X$ for the induced map. The pair  $(U, f^{-1}Z)$ admits an equivariant Nisnevich linearization and so by 
\aref{cor:EHP-reduction2},   \aref{thm:Purity} is true for $(U,f^{-1}Z)$. Therefore by  \aref{cor:EHP-reduction1} it is also true for $(X,Z)$. 

$\hfill \square$

\vskip .3cm

Using the same line of proof as for \aref{thm:Purity} verbatim, 
we obtain the following result for equivariant blow-ups.

\begin{thm}
\label{thm:Blow-up-square}
Let $k$ and $G$ be as in \autoref{thm:Purity}. 
Let $Z \inj X$ be a closed immersion in $\Sm^G_k$ with complement $U = X \setminus Z$ and
let $p: X' \to X$ denote the blow-up of $X$ along $Z$. 
Then the square
\begin{equation*}
\xymatrix@C1pc{
p^{-1}(Z) \ar[r] \ar[d] & {X'}/U \ar[d] \\
Z \ar[r] & X/U}
\end{equation*}
is a homotopy pushout square of motivic $G$-spaces.
\end{thm}

\begin{cor}
\label{cor:Blow-up-les}
With the same assumptions as in \aref{thm:Purity} and \aref{thm:Blow-up-square}, there are naturally induced homotopy equivalences of
equivariant $K$-theory spectra $\sK^G(X/U)\cong\sK^G(\Th(N_{Z/X}))$ and $\sK^G(({X'}/U) {\underset{p^{-1}(Z)}\coprod} Z)\cong\sK^G(X/U)$.
Moreover, 
there are long exact sequences 
\begin{equation*}
\cdots 
\rightarrow
K^G_{n+1}(U)
\rightarrow
K^G_{n}(\Th(N_{Z/X}))
\rightarrow
K^G_{n}(X)
\rightarrow
K^G_{n}(U)
\rightarrow
\cdots, 
\end{equation*}
and
\begin{equation*}
\cdots 
\rightarrow
K^G_{n}(X)
\rightarrow
K^G_{n}(Z)\oplus K^G_{n}({X'})
\rightarrow
K^G_{n}(p^{-1}(Z))
\rightarrow
K^G_{n-1}(X)
\rightarrow
\cdots.
\end{equation*}
\end{cor}

\section*{Acknowledgments.}
The first version of this paper was written when AK was visiting the Department of Mathematics at University of Oslo in the summer of 2011.
He thanks the department for the invitation and support. 
The paper was finished during the Special Semester in Motivic Homotopy Theory at Universit{\"at} Duisburg-Essen, 
and we are grateful for its hospitality and support. 
We thank Aravind Asok, David Gepner, Marc Hoyois, Marc Levine, and Ben Williams \cite{Williams} for useful discussions on subjects related to this paper.
Finally, we thank the referee for a careful reading of this paper.

\bibliographystyle{plain}
\bibliography{unstable}

\begin{thebibliography}{10}

\bibitem{SGA4Tome1}
{\em Th\'eorie des topos et cohomologie \'etale des sch\'emas. {T}ome 1:
  {T}h\'eorie des topos}.
\newblock Lecture Notes in Mathematics, Vol. 269. Springer-Verlag, Berlin-New
  York, 1972.
\newblock S{\'e}minaire de G{\'e}om{\'e}trie Alg{\'e}brique du Bois-Marie
  1963--1964 (SGA 4), Dirig{\'e} par M. Artin, A. Grothendieck, et J. L.
  Verdier. Avec la collaboration de N. Bourbaki, P. Deligne et B. Saint-Donat.

\bibitem{Asok}
A.~Asok.
\newblock {${\bf A}^1$-contractibility and topological contractibility}.
\newblock {\em "Group actions, generalized cohomology theories, and affine
  algebraic geometry", University of Ottawa, Canada}, 2011.

\bibitem{Barwick}
C.~Barwick.
\newblock On left and right model categories and left and right {B}ousfield
  localizations.
\newblock {\em Homology, Homotopy Appl.}, 12(2):245--320, 2010.

\bibitem{Blander}
B.~A. Blander.
\newblock Local projective model structures on simplicial presheaves.
\newblock {\em $K$-Theory}, 24(3):283--301, 2001.

\bibitem{handbook2}
F.~Borceux.
\newblock {\em Handbook of categorical algebra. 2}, volume~51 of {\em
  Encyclopedia of Mathematics and its Applications}.
\newblock Cambridge University Press, Cambridge, 1994.
\newblock Categories and structures.

\bibitem{CJ}
G.~{Carlsson} and R.~{Joshua}.
\newblock {Equivariant motivic homotopy theory}.
\newblock {\em ArXiv e-prints}, April 2014.

\bibitem{Cast}
D.~Castella.
\newblock Trivialit\'e des fibr\'es vectoriels \'equivariants pour les groupes
  ab\'eliens finis.
\newblock {\em Adv. Math.}, 151(1):36--44, 2000.

\bibitem{CR}
C.~Curtis and I.~Reiner.
\newblock {\em Representation theory of finite groups and associative
  algebras}.
\newblock Pure and Applied Mathematics, Vol. XI. Interscience Publishers, a
  division of John Wiley \& Sons, New York-London, 1962.

\bibitem{Deligne}
P.~Deligne.
\newblock Voevodsky's lectures on motivic cohomology 2000/2001.
\newblock In {\em Algebraic topology}, volume~4 of {\em Abel Symp.}, pages
  355--409. Springer, Berlin, 2009.

\bibitem{DHI}
D.~Dugger, S.~Hollander, and D.~Isaksen.
\newblock Hypercovers and simplicial presheaves.
\newblock {\em Math. Proc. Cambridge Philos. Soc.}, 136(1):9--51, 2004.

\bibitem{DRO}
B.~I. Dundas, O.~R{\"o}ndigs, and P.~A. {\O}stv{\ae}r.
\newblock Motivic functors.
\newblock {\em Doc. Math.}, 8:489--525 (electronic), 2003.

\bibitem{GoerssJardine}
P.~G. Goerss and J.~F. Jardine.
\newblock {\em Simplicial homotopy theory}, volume 174 of {\em Progress in
  Mathematics}.
\newblock Birkh\"auser Verlag, Basel, 1999.

\bibitem{Heller}
A.~Heller.
\newblock Homotopy theories.
\newblock {\em Mem. Amer. Math. Soc.}, 71(383):vi+78, 1988.

\bibitem{HVO}
J.~{Heller}, M.~{Voineagu}, and P.~A. {{\O}stv{\ae}r}.
\newblock Equivariant cycles and cancellation for motivic cohomology.
\newblock {\em Doc. Math.}, 20:269--332 (electronic), 2015.

\bibitem{Herrmann}
P.~{Herrmann}.
\newblock {Equivariant Motivic Homotopy Theory}.
\newblock {\em ArXiv e-prints}, December 2013.

\bibitem{HHR}
M.~A. {Hill}, M.~J. {Hopkins}, and D.~C. {Ravenel}.
\newblock {On the non-existence of elements of Kervaire invariant one}.
\newblock {\em ArXiv e-prints}, August 2009.

\bibitem{Hirsc}
P.~S. Hirschhorn.
\newblock {\em Model categories and their localizations}, volume~99 of {\em
  Mathematical Surveys and Monographs}.
\newblock American Mathematical Society, Providence, RI, 2003.

\bibitem{Hornbostel}
J.~Hornbostel.
\newblock Localizations in motivic homotopy theory.
\newblock {\em Math. Proc. Cambridge Philos. Soc.}, 140(1):95--114, 2006.

\bibitem{HKO}
M.~{Hoyois}, A.~{Krishna}, and P.~A. {{\O}stv{\ae}r}.
\newblock {${\bf A}^1$-contractibility of Koras-Russell threefolds}.
\newblock {\em ArXiv e-prints}, September 2014.

\bibitem{Isaksen:flasque}
D.~C. Isaksen.
\newblock Flasque model structures for simplicial presheaves.
\newblock {\em $K$-Theory}, 36(3-4):371--395 (2006), 2005.

\bibitem{Jardine:spre}
J.~F. Jardine.
\newblock Simplicial presheaves.
\newblock {\em J. Pure Appl. Algebra}, 47(1):35--87, 1987.

\bibitem{Jardine:genet}
J.~F. Jardine.
\newblock {\em Generalized \'etale cohomology theories}, volume 146 of {\em
  Progress in Mathematics}.
\newblock Birkh\"auser Verlag, Basel, 1997.

\bibitem{Knutson}
D.~Knutson.
\newblock {\em Algebraic spaces}.
\newblock Lecture Notes in Mathematics, Vol. 203. Springer-Verlag, Berlin-New
  York, 1971.

\bibitem{Kollar}
J.~Koll{\'a}r.
\newblock Quotient spaces modulo algebraic groups.
\newblock {\em Ann. of Math. (2)}, 145(1):33--79, 1997.

\bibitem{AO}
A.~Krishna and P.~A. {\O}stv{\ae}r.
\newblock Nisnevich descent for {$K$}-theory of {D}eligne-{M}umford stacks.
\newblock {\em J. K-Theory}, 9(2):291--331, 2012.

\bibitem{Lurie}
J.~Lurie.
\newblock {\em Higher topos theory}, volume 170 of {\em Annals of Mathematics
  Studies}.
\newblock Princeton University Press, Princeton, NJ, 2009.

\bibitem{MMP}
M.~Masuda, L.~Moser-Jauslin, and T.~Petrie.
\newblock The equivariant {S}erre problem for abelian groups.
\newblock {\em Topology}, 35(2):329--334, 1996.

\bibitem{MV}
F.~Morel and V.~Voevodsky.
\newblock {${\bf A}^1$}-homotopy theory of schemes.
\newblock {\em Inst. Hautes \'Etudes Sci. Publ. Math.}, (90):45--143 (2001),
  1999.

\bibitem{Rydh}
D.~Rydh.
\newblock Existence and properties of geometric quotients.
\newblock {\em J. Algebraic Geom.}, 22(4):629--669, 2013.

\bibitem{Serpe}
C.~{Serpe}.
\newblock {Descent properties of equivariant {$K$}-theory}.
\newblock {\em ArXiv e-prints}, February 2010.

\bibitem{Suslin:Grayson}
A.~Suslin.
\newblock On the {G}rayson spectral sequence.
\newblock {\em Tr. Mat. Inst. Steklova}, 241(Teor. Chisel, Algebra i Algebr.
  Geom.):218--253, 2003.

\bibitem{Thomason:hocolim}
R.~W. Thomason.
\newblock Homotopy colimits in the category of small categories.
\newblock {\em Math. Proc. Cambridge Philos. Soc.}, 85(1):91--109, 1979.

\bibitem{Thomason:trace}
R.~W. Thomason.
\newblock Lefschetz-{R}iemann-{R}och theorem and coherent trace formula.
\newblock {\em Invent. Math.}, 85(3):515--543, 1986.

\bibitem{Thomason:GK}
R.~W. Thomason.
\newblock Algebraic {$K$}-theory of group scheme actions.
\newblock In {\em Algebraic topology and algebraic {$K$}-theory ({P}rinceton,
  {N}.{J}., 1983)}, volume 113 of {\em Ann. of Math. Stud.}, pages 539--563.
  Princeton Univ. Press, Princeton, NJ, 1987.

\bibitem{Thomason:resolution}
R.~W. Thomason.
\newblock Equivariant resolution, linearization, and {H}ilbert's fourteenth
  problem over arbitrary base schemes.
\newblock {\em Adv. in Math.}, 65(1):16--34, 1987.

\bibitem{ThomasonDuke56}
R.~W. Thomason.
\newblock Equivariant algebraic vs.\ topological {$K$}-homology
  {A}tiyah-{S}egal-style.
\newblock {\em Duke Math. J.}, 56(3):589--636, 1988.

\bibitem{Vistoli:finite}
A.~Vistoli.
\newblock Higher equivariant {$K$}-theory for finite group actions.
\newblock {\em Duke Math. J.}, 63(2):399--419, 1991.

\bibitem{Voevodsky:Z2}
V.~Voevodsky.
\newblock Motivic cohomology with {${\bf Z}/2$}-coefficients.
\newblock {\em Publ. Math. Inst. Hautes \'Etudes Sci.}, (98):59--104, 2003.

\bibitem{V:cd}
V.~Voevodsky.
\newblock Homotopy theory of simplicial sheaves in completely decomposable
  topologies.
\newblock {\em J. Pure Appl. Algebra}, 214(8):1384--1398, 2010.

\bibitem{V:Niscdh}
V.~Voevodsky.
\newblock Unstable motivic homotopy categories in {N}isnevich and
  cdh-topologies.
\newblock {\em J. Pure Appl. Algebra}, 214(8):1399--1406, 2010.

\bibitem{Voevodsky:Zl}
V.~Voevodsky.
\newblock On motivic cohomology with {${\bf Z}/l$}-coefficients.
\newblock {\em Ann. of Math. (2)}, 174(1):401--438, 2011.

\bibitem{Walker}
M.~E. Walker.
\newblock {Motivic complexes and the {$K$}-theory of automorphisms}.
\newblock {K-theory Preprint Archives, no. 0205}.

\bibitem{Williams}
B.~Williams.
\newblock {The intermediate equivariant Nisnevich topology}.
\newblock {\em Personal Note}.

\end{thebibliography}

\end{document}